\documentclass{amsart}

\usepackage{color}
\usepackage{amsfonts}
\usepackage{amssymb}
\usepackage{enumerate}
\usepackage{amsmath}
\usepackage{amsthm}
\usepackage{graphicx}
\usepackage{array}
\usepackage[english]{babel}

\numberwithin{equation}{section}
\allowdisplaybreaks

\newtheorem{theorem}{Theorem}[section]
\newtheorem{lemma}[theorem]{Lemma}

\newtheorem{problem}[theorem]{Problem}

\newtheorem{corollary}[theorem]{Corollary}
\newtheorem{statement}[theorem]{Statement}
\newtheorem{fact}[theorem]{Fact}
\newtheorem*{ind}{Theorem \ref{t:counter}}
\newtheorem*{arb1}{Theorem \ref{t:arb1}}
\newtheorem*{arb2}{Theorem \ref{t:arb2}}

\newtheorem*{MDS2}{Theorem \ref{t:MDS2}}
\newtheorem*{cMDS}{Theorem \ref{t:cMDS}}
\newtheorem*{c:c}{Corollary \ref{c:c}}

\theoremstyle{definition}

\newtheorem{remark}[theorem]{Remark}
\newtheorem{definition}[theorem]{Definition}

\DeclareMathOperator{\Cov}{Cov}

\newcommand{\N}{\mathbb{N}}

\newcommand{\R}{\mathbb{R}}

\renewcommand{\P}{\mathop{\mathbb{P{}}}\nolimits}
\newcommand{\E}{\mathop{\mathbb{E{}}}\nolimits}
\newcommand{\eps}{\varepsilon}

\newcommand{\iH}{\mathcal{H}}

\newcommand{\iF}{\mathcal{F}}

\begin{document}

\title{Baum--Katz type theorems with exact threshold}

\author{Rich\'ard Balka}

\address{Department of Mathematics, University of British Columbia, and Pacific Institute for the Mathematical Sciences, Vancouver, BC V6T 1Z2, Canada}

\thanks{The first author was supported by the National Research, Development and Innovation Office--NKFIH, Grant 104178.}

\email{balka@math.ubc.ca}

\author{Tibor T\'om\'acs}

\address{Eszterh\'azy K\'aroly University of Applied Sciences, Institute of Mathematics and Informatics, Eger, H-3300 Hungary}

\email{tomacs.tibor@uni-eszterhazy.hu}



\begin{abstract} Let $\{X_n\}_{n\geq 1}$ be either a sequence of arbitrary random variables, or a martingale difference sequence, or a centered sequence with a suitable level of negative dependence. We prove Baum--Katz type theorems by only assuming that the variables $X_n$ satisfy a uniform moment bound condition. We also prove that this condition is best possible even for sequences of centered, independent random variables. This leads to Marcinkiewicz--Zygmund type strong laws of large numbers with estimate for the rate of convergence.   
\end{abstract}

\maketitle 

 \section{Introduction} 
 
 \subsection{Motivation and related results} 
 Let $\{X_n\}_{n\geq 1}$ be a sequence of random variables, we always assume that they are defined on the same probability space. For all $n\in \N^+$ let $S_n=\sum_{i=1}^n X_i$ and $M_n=\max_{1\leq i\leq n} |S_i|$. For some positive parameters $p,r$ consider the statement 
 \begin{equation} \label{eq:main} \sum_{n=1}^{\infty} n^{p/r-2} \P(M_n>\eps n^{1/r})<\infty \quad \textrm{for all } \eps>0, \tag{M}
 \end{equation}
 and the weaker claim
 \begin{equation} \label{eq:main2} \sum_{n=1}^{\infty} n^{p/r-2} \P(|S_n|>\eps n^{1/r})<\infty \quad \textrm{for all } \eps>0. \tag{S}
 \end{equation}
 The main goal of the paper is to prove \eqref{eq:main} or \eqref{eq:main2} 
 under different conditions. We may assume that $0<r<2$ and $p\geq r$. Indeed, if $p<r$ then \eqref{eq:main} trivially holds. If $p\geq r\geq 2$ then by the central limit theorem the sum in \eqref{eq:main2} is divergent for all $\eps>0$ even if $\{X_n\}_{n\geq 1}$ is an i.i.d.\ sequence with mean zero and finite variance. 
 
 We cite the known results in this subsection, for the new ones see Subsection~\ref{ss:res}. First consider the classical results for i.i.d.\ random variables. Following Hsu and Robbins \cite{HR} we say that a sequence $\{X_n\}_{n\geq 1}$ \emph{converges completely} to $0$ if 
 \[\sum_{n=1}^{\infty} \P(|X_n|>\eps)<\infty \quad \textrm{for all } \eps>0.\]
 By the Borel--Cantelli Lemma this implies that $X_n\to 0$ almost surely, but the converse is not necessarily true. If $\{X_n\}_{n\geq 1}$ is a centered i.i.d.\ sequence of random variables then $S_n/n\to 0$ almost surely by the strong law of large numbers. Under what conditions does $S_n/n$ converge completely to $0$? Hsu and Robbins \cite{HR} showed that $\E(X_1^2)<\infty$ is sufficient, and Erd\H{o}s \cite{E,E2} proved that it is necessary.   
 
\begin{theorem}[Hsu--Robbins--Erd\H{o}s strong law of large numbers] Let $\{X_n\}_{n\geq 1}$ be a sequence of centered i.i.d.\ random variables. Then the following are equivalent:
\begin{enumerate}[(i)]
 \item \label{E1} $\E(X_1^2)<\infty$, 
 \item \label{E2} $\sum_{n=1}^{\infty} \P(|S_n|>\eps n)<\infty$ for all $\eps>0$. 
\end{enumerate} 
\end{theorem} 
 
A more general classical theorem is the following.
  
\begin{theorem}[Baum--Katz, Chow] \label{t:BK} 
Let $\{X_n\}_{n\geq 1}$ be a sequence of i.i.d.\ random variables. Let $0<r<2$ and let $p\geq r$. The following statements are equivalent:
\begin{align}
\label{al1} &\E(|X_1|^p)<\infty \textrm{ and if }  p\geq 1 \textrm{ then } \E(X_1)=0, \tag{1}\\
\label{al2} &\sum_{n=1}^{\infty} n^{p/r-2}
\P(|S_n|>\eps n^{1/r})<\infty \quad \textrm{for all } \eps>0, \tag{2} \\ 
\label{al3} &\sum_{n=1}^{\infty} n^{p/r-2}
\P(M_n>\eps n^{1/r})<\infty \quad \textrm{for all } \eps>0. \tag{3} 
\end{align}    
\end{theorem}

The equivalence of \eqref{al1} and \eqref{al2} in the case $r=p=1$ is due to Spitzer \cite{S}, the case $r=1$, $p=2$ is the Hsu--Robbins--Erd\H{o}s strong law, while 
the general case is due to Baum and Katz \cite{BK}. For the equivalence of \eqref{al1} and \eqref{al3} see Chow \cite{C}. 

The next theorem is the Marcinkiewicz--Zygmund strong law of large numbers, see \cite{MZ}. Note that the case $p=1$ dates back to Kolmogorov~\cite{Ko} and includes the classical strong law of large numbers. 

\begin{theorem}[Marcinkiewicz--Zygmund strong law of large numbers] \label{t:MZ} For an i.i.d.\ sequence of random variables $\{X_n\}_{n\geq 1}$ and $0<p<2$ the following are equivalent:
\begin{enumerate}[(i)]
\item $\E(|X_1|^p)<\infty$ and if $p\geq 1$ then $\E(X_1)=0$,
\item $\lim_{n\to \infty} n^{-1/p}S_n=0$ almost surely.
\end{enumerate}
\end{theorem} 

The following statement explains the connection between 
Theorem~\ref{t:BK} and the Marcinkiewicz--Zygmund strong law of large numbers and its rate of convergence. 
It is formulated for arbitrary sequences of random variables, see \cite[Remarks~1~and~2]{DM} and see also \cite[Lemma~4]{L} for the proof of part \eqref{eq:2nd}.

\begin{statement} \label{st:1} Let $\{X_n\}_{n\geq 1}$ be an arbitrary sequence of random variables, let $0<r<2$ and let $p\geq r$. Assume that \eqref{eq:main} holds. 
\begin{enumerate}[(i)]
\item \label{eq:1st} 
If $p=r$ then for all $\eps>0$ we have \[\sum_{n=1}^{\infty} \P(M_{2^n}>\eps 2^{n/p})<\infty,\] 
which implies that $\lim_{n\to \infty} n^{-1/p} S_n=0$ almost surely.  
\item \label{eq:2nd} If $p>r$ then for all $\eps>0$ we have
\[\sum_{n=1}^{\infty} n^{p/r-2} \P\left(\sup_{k\geq n} k^{-1/r}|S_k|>\eps\right)<\infty.\]
Since the above probabilities are non-increasing, we obtain that
\[\P\left(\sup_{k\geq n} k^{-1/r}|S_k|>\eps\right)=o(n^{1-p/r}) \quad \textrm{ as } n\to \infty.\] 
\end{enumerate}
\end{statement}

In contrast to Theorems~\ref{t:BK} and~\ref{t:MZ}, we will not assume independence or identical distributions in the following. Now we summarize the known results in this direction, which requires some technical definitions. 

The following theorem partly generalizes the Marcinkiewicz--Zygmund strong law of large numbers, see Stout~\cite[Theorem~3.3.9]{Stout} and \cite[Corollary 3.3.5]{Stout}. It is based on Chung~\cite{Ch}
in the case of independent variables and is implicitly contained in 
Lo\`eve~\cite{Lo}. Its proof uses a conditional three series theorem \cite[Theorem~2.8.8]{Stout}. 

\begin{theorem}[Stout] \label{t:Stout}
	Let $f\colon [0,\infty)\to \R^+$ be a non-decreasing function with \[\sum_{n=1}^{\infty} \frac{1}{f(2^n)}<\infty.\] 
	Let $0<p<2$. Let $\{X_n\}_{n\geq 1}$ be  
	\begin{enumerate}[(i)]
		\item an arbitrary sequence of random variables if $0<p<1$ and suppose that $x\mapsto x^{p-1}f(x)$ is non-increasing,
		\item a martingale difference sequence if $1\leq p<2$ and assume that $x\mapsto x^{p-2}f(x)$ is non-increasing.
	\end{enumerate} 
	If $\sup_{n\geq 1} \E(|X_n|^p f(|X_n|))<\infty$ then $\lim_{n\to \infty} n^{-1/p}S_n=0$ almost surely. 
\end{theorem}

\begin{definition}
We say that the sequence $\{X_n\}_{n\geq 1}$ is 
\emph{weakly dominated} by a random variable $X$ if there is a constant $C\in \R^+$ such that for all $n\in \N^+$ and $x>0$ we have
\begin{equation} \label{WD} 
\P(|X_n|>x)\leq C\P(|X|>x), \tag{WD}
\end{equation}
and \emph{weakly mean dominated} if  
\begin{equation} \label{WMD} 
\frac{1}{n} \sum_{k=1}^{n} \P(|X_k|>x)\leq C \P(|X|>x) \tag{WMD}
\end{equation}
for some $C\in \R^+$ and for all $n\in \N^+$ and $x>0$. 
\end{definition}
The following definition was introduced by Alam and Saxena \cite{AS} and Joag-Dev and Proschan~\cite{JP}. 

\begin{definition} 
	A finite family of random variables 
	$\{X_i: 1\leq i\leq n\}$ is called \emph{negatively associated} (NA) if for every pair of disjoint subsets $A_1$, $A_2\subset \{1,\dots, n\}$ we have 
	\[\Cov(f_1(X_i: i\in A_1),\, f_2(X_j: j\in A_2))\leq 0\]
	for all coordinatewise non-decreasing functions $f_1$ and $f_2$ for which the covariance exists. An infinite family of random variables is NA if every finite subfamily is NA. 
\end{definition}

The following definition is due to Lehmann~\cite{Le}. 

\begin{definition}
Two random variables $X$ and $Y$ are called \emph{negatively quadrant dependent} (NQD) if for all $x,y\in \R$ we have 
\[ \P(X\leq x,\, Y\leq y)\leq \P(X\leq x)\P(Y\leq y).\]
\end{definition}

Every independent sequence is NA, and each pairwise independent sequence is pairwise NQD. It is proved in \cite{JP} that every NA sequence is pairwise NQD. 
For the next theorem see Kuczmaszewska \cite[Theorem~2.1]{K} and its proof.

\begin{theorem}[Kuczmaszewska] \label{t:K}
	Let $0<r<2$ and $p>r$. Assume that the sequence 
	$\{X_n\}_{n\geq 1}$ is weakly mean dominated by a random variable $X$ satisfying $\E(|X|^p)<\infty$. If $p>1$ assume that $\{X_n\}_{n\geq 1}$ is centered and NA. Then \eqref{eq:main} holds. 
\end{theorem}

\begin{remark} If $p=1$ we obtain the above theorem by applying the Markov inequality $\P(Z>t)\leq \E(Z^q)/t^q$ at the beginning of \cite[(2.4)]{K} for some $q>1$. 
\end{remark}

For the following theorem see Tan, Wang, and Zhang~\cite[Theorems~1.1,~1.2]{TWZ}, and see also Gan and Chen \cite[Theorem~2.2]{GC} for the second part of the statement. Note that it is strongly based on the pioneering work of Wu~\cite{W}.   

\begin{theorem}[Wu,~Tan--Wang--Zhang,~Gan--Chen] \label{t:W} Let $0<r\leq p$ and $1\leq p<2$. Let $\{X_n\}_{n\geq 1}$ be a centered, pairwise NQD sequence which is weakly dominated by a random variable $X$ with $\E(|X|^p)<\infty$. Then \eqref{eq:main2} holds. If $r\neq p$ then \eqref{eq:main} holds.
\end{theorem}

Now we state the last two results of this subsection, which consider martingale difference sequences (MDS). Miao, Yang, and Stoica \cite[Theorems~2.1\,(1),~2.3]{MYS} proved the following theorem about the case $1\leq p<2$. 

\begin{theorem}[Miao--Yang--Stoica] \label{t:KMYS} 
Assume that $0<r\leq p<2$. Let $\{X_n\}_{n\geq 1}$ be a MDS which is weakly mean dominated by $X$. Property \eqref{eq:main} holds if
\begin{enumerate}[(i)]
\item \label{K2} $r=p=1$ and $\E(|X| \log^+|X|)<\infty$,
\item \label{K3} $1<p<2$ and $\E(|X|^p)<\infty$. 
\end{enumerate}
\end{theorem}

\begin{remark} Note that \eqref{K2} is optimal: Elton \cite{El} proved that if $X$ is a centered random variable with $\E(|X|\log^+ |X|)=\infty$ then there is a martingale difference sequence $\{X_n\}_{n\geq 1}$ such that $X_n$ and $X$ have the same distribution for all $n$ and $S_n/n\to \infty$ as $n\to \infty$ almost surely. Thus \eqref{eq:main} cannot hold by  Statement~\ref{st:1}~\eqref{eq:1st}.
\end{remark}

The above theorem generalizes a result of Dedecker and Merlev\'ede \cite[Theorem~5]{DM} for real valued random variables. In the case $p\geq 2$ a new phenomenon emerges. For the following theorem see the proofs of \cite[Theorems~2.2,~2.4~(3)]{MYS}. 
\begin{theorem}[Miao--Yang--Stoica] \label{t:MYS}
Assume that $0<r<2\leq p$ and $q(r,p)=2(p-r)/(2-r)$. 
Let $\{X_n\}_{n\geq 1}$ be a MDS such that $\sup_{n\geq 1} \E(|X_n|^q)<\infty$. 
\begin{enumerate}[(i)]
\item If $q>q(r,p)$ then \eqref{eq:main} holds. 
\item If $q=q(p,r)$ then there is a MDS $\{X_n\}_{n\geq 1}$ which is weakly dominated by a random variable $X$ satisfying $\E(|X|^q)<\infty$ such that \eqref{eq:main2} does not hold. 	 
\end{enumerate}
\end{theorem}

\subsection{The results of the paper}   \label{ss:res}

The goal of the paper is to investigate statements \eqref{eq:main} and \eqref{eq:main2} for arbitrary random variables,  martingale difference sequences, and centered sequences with a certain level of negative dependence. We will deduce \eqref{eq:main} by assuming only a uniform moment condition $\sup_{n\geq 1} \E(|X_n|^q f(|X_n|))<\infty$, so in contrast to Theorems~\ref{t:K},~\ref{t:W}, and \ref{t:KMYS} properties \eqref{WD} and \eqref{WMD} will not be assumed. We will find the smallest possible suitable constant $q=q(p,r)$ which we call the \emph{critical exponent}. In particular, we generalize Theorems~\ref{t:MYS} and \ref{t:Stout}. We will be also able to determine the precise smaller order term $f$. Similarly to Theorem~\ref{t:Stout} the function $f\colon [0,\infty)\to \R^+$ might be any non-decreasing function satisfying $\sum_{n=1}^{\infty} 1/f(2^n)<\infty$, and the finiteness of the sum is really necessary even for \eqref{eq:main2}. By Corollary~\ref{c:f} we may assume that $f(n)=n^{o(1)}$ as $n\to \infty$, see also Remark~\ref{r:f} for the least possible order of magnitude of $f$. Stoica claimed similar theorems for martingale difference sequences in \cite{St1} and \cite{St2}, but those results are incorrect. He stated in \cite{St1} that if $0<r<2<p$ and $\{X_n\}_{n\geq 1}$ is a MDS with $\sup_{n\geq 1} \E(|X_n|^p)<\infty$ then \eqref{eq:main2} holds. This was disproved in \cite{S}, see Theorem~\ref{t:MYS}. Theorems~1 and~2 in \cite{St2} state that if $1\leq r\leq p<2$ and $\{X_n\}_{n\geq 1}$ is a MDS with $\sup_{n\geq 1} \E(|X_n|^p \log^+ |X_n|)<\infty$ then \eqref{eq:main2} holds. Theorem~\ref{t:counter} below witnesses that this is not true even for independent, centered sequences of random variables. 

\bigskip

The following theorem is one of the most important results in the paper. 

 \begin{theorem} \label{t:main} Let $f\colon [0,\infty)\to \R^+$ be a non-decreasing function such that \[\sum_{n=1}^{\infty} \frac{1}{f(2^n)}<\infty.\]
 Let $0<r<2$ and let $p\geq r$. Let $\{X_n\}_{n\geq 1}$ be a  
 \begin{enumerate}[(i)]
\item \label{b1} sequence of arbitrary random variables if $0<r\leq p\leq 1$ and $r<1$,
\item \label{b2} martingale difference sequence if $1<p\leq 2$ or $r=p=1$, 
\item \label{b3} centered, negatively associated sequence of random variables if $p\geq 2$. 
\end{enumerate} 	
If $\sup_{n\geq 1} \E(|X_n|^p f(|X_n|))<\infty$ then for all $\eps>0$ we have \[\sum_{n=1}^{\infty} n^{p/r-2} \P(M_n>\eps n^{1/r})<\infty.\]
 \end{theorem}
 
We will prove the above theorem in several steps. Theorem~\ref{t:arb1} implies \eqref{b1}, Theorems~\ref{t:MDS} and~\ref{t:MDS2} yield \eqref{b2}, and \eqref{b3} is stated as Theorem~\ref{t:indep}. 

The next theorem is analogous to Theorem~\ref{t:W} with a similar proof. Thus, instead of proving it, we suggest the reader to follow \cite[Theorems~1.1,~1.2]{TWZ}. 

\begin{theorem}\label{t:pind}
	Let $f\colon [0,\infty)\to \R^+$ be a non-decreasing function such that \[\sum_{n=1}^{\infty} \frac{1}{f(2^n)}<\infty.\]
	Let $0<r\leq p$ and $1\leq p<2$. If $\{X_n\}_{n\geq 1}$ is a centered, pairwise NQD sequence with $\sup_{n\geq 1} \E(|X_n|^pf(|X_n|))<\infty$ then \eqref{eq:main2} holds. 
	If $r\neq p$ then \eqref{eq:main} holds, too. 
\end{theorem}

The following theorem shows that the moment conditions above are sharp even for independent, centered random variables, even if $r=p$ and we want to obtain only $\lim_{n\to \infty} n^{-1/p} S_n= 0$ almost surely, recall Statement~\ref{st:1}~\eqref{eq:1st}.
 
 \begin{ind} Let $f\colon [0,\infty)\to \R^+$ be a non-decreasing function such that \[\sum_{n=1}^{\infty} \frac{1}{f(2^n)}=\infty.\]
 Let $0<r<2$ and let $p\geq r$. Then there exists a sequence of independent, centered random variables $\{X_n\}_{n\geq 1}$
 such that $\sup_{n\geq 1} \E(|X_n|^p f(|X_n|))<\infty$ and
 \[\sum_{n=1}^{\infty} n^{p/r-2} \P(|S_n|> n^{1/r})=\infty.\]
 Moreover, if $r=p$ then $\limsup_{n\to \infty} n^{-1/p}S_n\geq 1$ almost surely. 
 \end{ind}
 
First consider Theorem~\ref{t:main}~\eqref{b1}. 
In the case of arbitrary random variables 
we need to suppose that $r<1$. Indeed, for $1\leq r\leq p$ let $X_n\equiv 1$ for all $n$. 
Then $\sup_{n\geq 1} \E(|X_n|^q)=1$ for all $q>0$ but
\[\sum_{n=1}^{\infty} n^{p/r-2} \P(|S_n|>(1/2)n^{1/r})=\sum_{n=1}^{\infty} n^{p/r-2}\geq \sum_{n=1}^{\infty} n^{-1}=\infty.\] 
Theorem~\ref{t:main}~\eqref{b1} easily follows from the following, more general theorem. 

\begin{arb1}  Let $f\colon [0,\infty)\to \R^+$ be a non-decreasing function such that
	\[\sum_{n=1}^{\infty} \frac{1}{f(2^n)}<\infty.\]
	Let $0<r<1$ and let $p\geq r$, and define $q=q(r,p)=\max\{p, (p-r)/(1-r)\}$. 
	Assume that $\{X_n\}_{n\geq 1}$ is a sequence of random variables with $\sup_{n\geq 1} \E(|X_n|^q f(|X_n|))<\infty$. Then for all $\eps>0$ we have
	\[\sum_{n=1}^{\infty} n^{p/r-2} \P(M_n>\eps n^{1/r})<\infty.\]
\end{arb1}

We prove that the above theorem is sharp. By Theorem~\ref{t:counter} it is enough to consider the case $p\geq 1$. 

\begin{arb2} Let $f\colon [0,\infty)\to \R^+$ be a non-decreasing function such that
	\[\sum_{n=1}^{\infty} \frac{1}{f(2^n)}=\infty.\]
	Let $0<r<1\leq p$ and let $q=q(r,p)=(p-r)/(1-r)$. Then there is a sequence of random variables $\{X_n\}_{n\geq 1}$ such that $\sup_{n\geq 1} \E(|X_n|^q f(|X_n|))<\infty$ and
	\[\sum_{n=1}^{\infty} n^{p/r-2} \P(|S_n|>n^{1/r})=\infty.\]
\end{arb2}

We prove Theorem~\ref{t:main}~\eqref{b2} for $p<2$ as Theorem~\ref{t:MDS}. We follow the strategy of the proof of \cite[Appendix~A.1]{DM}. Theorem~\ref{t:arb2} yields that Theorem~\ref{t:main}~\eqref{b2} does not remain true for arbitrary random variables. 
 
The following theorems handle martingale difference sequences in the case $p\geq 2$. In particular, the next theorem proves Theorem~\ref{t:main}~\eqref{b2} for $p=2$. 

\begin{MDS2} Let $f\colon [0,\infty)\to \R^+$ be a non-decreasing function such that
	\[\sum_{n=1}^{\infty} \frac{1}{f(2^n)}<\infty.\]
	Let $0<r<2\leq p$ and let $q=q(r,p)=2(p-r)/(2-r)$. Let $\{X_n\}_{n\geq 1}$ be a martingale difference sequence such that $\sup_{n\geq 1} \E(|X_n|^q f(|X_n|))<\infty$. Then for all $\eps>0$ we have \[\sum_{n=1}^{\infty} n^{p/r-2} \P(M_n>\eps n^{1/r})<\infty.\]
\end{MDS2}

The following theorem witnesses that the above result is best possible.

\begin{cMDS} Let $f\colon [0,\infty)\to \R^+$ be a non-decreasing function such that
	\[\sum_{n=1}^{\infty} \frac{1}{f(2^n)}=\infty.\]
	Let $0<r<2\leq p$ and let $q=q(r,p)=2(p-r)/(2-r)$. Then there is a martingale difference sequence $\{X_n\}_{n\geq 1}$ such that 
	$\sup_{n\geq 1} \E(|X_n|^q f(|X_n|))<\infty$ and 
	\[\sum_{n=1}^{\infty} n^{p/r-2} \P(|S_n|>n^{1/r})=\infty.\]
\end{cMDS}
 
Miao, Yang, and Stoica proved that the threshold in Theorems~\ref{t:MDS2} and~\ref{t:cMDS} is at $q(r,p)=2(p-r)/(2-r)$, recall Theorem~\ref{t:MYS}. We will improve their methods in order to find the precise smaller order term.
 
Theorem~\ref{t:main}~\eqref{b3} is stated as Theorem~\ref{t:indep}, which will simply follow from an inequality of Shao~\cite{Sh}. If $0<r<1$ then we can remove the assumption that $\{X_n\}_{n\geq 1}$ is centered from Theorems~\ref{t:pind} and~\ref{t:indep}.
 
 \begin{c:c} 
 Let $f\colon [0,\infty)\to \R^+$ be a non-decreasing function such that 
 \[\sum_{n=1}^{\infty} \frac{1}{f(2^n)}<\infty.\]
 Let $0<r<1\leq p$, and let $\{X_n\}_{n\geq 1}$ be a sequence of 
 \begin{enumerate}
 \item pairwise NQD random variables if $1\leq p<2$,
 \item negatively associated random variables if $p\geq 2$.
 \end{enumerate} 
 Assume that $\sup_{n\geq 1} \E(|X_n|^pf(|X_n|))<\infty$. Then for all $\eps>0$ we have \[\sum_{n=1}^{\infty} n^{p/r-2} \P(M_n>\eps n^{1/r})<\infty.\]
 \end{c:c}
 
The above theorems witness that 
if $0<r<2<p$ then the critical exponents for independent centered sequences and martingale difference sequences are different, since $p<2(p-r)/(2-r)$. See Table~\ref{table} for the values of the critical exponents.

\begin{table}[!h]
\caption{The critical exponents for different intervals of $p$ and types of sequences. ICS and MDS denote independent, centered sequences, and martingale difference sequences, respectively.}
\label{table} 
\setlength{\extrarowheight}{2pt}
\centering
\begin{tabular}{|c|c|c|c|} 
\hline
\mbox{} & ICS & MDS & Arbitrary Sequences \\
\cline{1-4} 
$p\leq 1$ & $p$ & $p$ & $p$ if $r<1$\\
\cline{1-4}
$1<p\leq 2$ & $p$ & $p$ & $(p-r)/(1-r)$ if $r<1$ \\
\cline{1-4}
$p>2$ & $p$ & $2(p-r)/(2-r)$ &  $(p-r)/(1-r)$ if $r<1$  \\
\hline
\end{tabular}
\end{table}

We do not know much about pairwise independent random variables if $p\geq 2$.  

\begin{problem} 
Let $0<r<2\leq  p$ and let $k\geq 2$ be an integer. Let $\{X_n\}_{n\geq 1}$ be a sequence of $k$-wise independent, centered random variables. Do there exist results similar to Theorem~\ref{t:MDS2} (replace $M_n$ by $|S_n|$ if necessary) and Theorem~\ref{t:cMDS} with some $q=q(r,p,k)$? If yes, is it true that $q(r,p,k)=p$ for all $r,p,k$?   
\end{problem}

In fact, for some values of $r,p,k$ we can show that $q(r,p,k)=p$ is the critical exponent for \eqref{eq:main2}.  
The following remark is about the least possible order of magnitude of $f$ in the above theorems.

 \begin{remark} \label{r:f} Let $\log^+(x)=\max\{1,\log x\}$ for $x>0$ and $\log^+(0)=1$. For $k\in \N^+$ let $\log^+_{k}(x)$ denote the $k$th iteration of $\log^+ (x)$. For  $m\in \N^{+}$ and $\eps> 0$ define the functions $f_{m},f_{m,\eps}\colon [0,\infty) \to \R^+$ as
 	\begin{align*} 
 	f_{m}(x)&=\prod_{k=1}^{m} \log^+_k (x), \\
 	f_{m,\eps}(x)&=f_{m}(x) \left(\log_{m}^+(x)\right)^{\eps}.
 	\end{align*}
 	It is easy to see that for all $m\in \N^+$ and $\eps>0$ we have
 	\[\sum_{n=1}^{\infty}  \frac{1}{f_{m}(2^n)}=\infty \quad  \textrm{and} \quad \sum_{n=1}^{\infty} \frac{1}{f_{m,\eps}(2^n)}<\infty.\]
 \end{remark}
 
 In Section~\ref{s:pr} we recall some definitions and easy facts. In Section~\ref{s:tl} we prove a number of technical lemmas. Section~\ref{s:arb} is devoted to arbitrary random variables. In Section~\ref{s:MDS} we prove our theorems about martingale difference sequences. Finally, in Section~\ref{s:ind} we verify Theorems~\ref{t:indep} and~\ref{t:counter}, and Corollary~\ref{c:c}. Note that the proofs after Section~\ref{s:tl} can be read independently of each other.

 \section{Preliminaries} \label{s:pr} 
 
Let $\{X_n\}_{n\geq 1}$ be a sequence of random variables defined on the probability space $(\Omega,\iF, \P)$. It is a \emph{martingale difference sequence} if there is a filtration $\{\iF_n\}_{n\geq 0}$ such that $\iF_{0}=\{\emptyset, \Omega\}$, $X_n$ is measurable with respect to $\iF_n$, and $\E(X_n \, |\, \iF_{n-1})=0$ for all $n\in \N^+$. We may assume without loss of generality that $\iF_n=\sigma(X_1,\dots,X_n)$ is the $\sigma$-algebra generated by $X_1,\dots,X_n$ for all $n\in \N^+$. A random variable is called \emph{centered} if $\E(X)=0$. 

 Let $E\subset \R$ and let $f\colon E\to \R$. We say that $f$ is \emph{non-decreasing} (or \emph{increasing}) if for all $x,y\in E$, $x<y$ we have $f(x)\leq f(y)$ (or $f(x)<f(y)$). We can similarly define the notions \emph{non-increasing} and \emph{decreasing}, and if $E=\N^+$ then our definitions extend to sequences as well. 
Let $I(A)$ denote the indicator function of an event $A$. We use the notation $a\lesssim b$ if $a\leq cb$ with some $c\in \R^+$, where $c$ depends only on earlier fixed constants. The notation $a_n=o(b_n)$ as $n\to \infty$ means that $\lim_{n\to \infty} a_n/b_n=0$. We need the following facts.
   
\begin{fact} \label{f:1}
Let $f\colon [0,\infty)\to \R^+$ be a non-decreasing function. Then the following statements are equivalent:
\begin{enumerate} [(i)]
\item  \label{e:f1} $\sum_{n=1}^{\infty} 1/f(2^{cn})<\infty$ for some $c>0$,
\item \label{e:f2} $\sum_{n=1}^{\infty} 1/f(\eps 2^{cn})<\infty$ for all $\eps,c>0$,
\item \label{e:f3} $\sum_{n=1}^{\infty} 1/(nf(n^c))<\infty$ for some $c>0$,
\item \label{e:f4} $\sum_{n=1}^{\infty} 1/(nf(\eps n^c))<\infty$ for all $\eps, c>0$.
\end{enumerate}
\end{fact}

The equivalence $\eqref{e:f1} \Leftrightarrow \eqref{e:f3}$ above follows from the equiconvergence of the series $\sum_{n=1}^{\infty} a_n $ and  $\sum_{n=1}^{\infty} 2^n a_{2^n}$ for any non-increasing, positive  sequence $\{a_n\}_{n\geq 1}$. Easy comparison implies the equivalences $\eqref{e:f1} \Leftrightarrow \eqref{e:f2}$ and $\eqref{e:f3} \Leftrightarrow \eqref{e:f4}$.
 
\begin{fact} \label{f:2} 	Let $\{X_n\}_{n\geq 1}$ be a sequence of random variables and let $g,h\colon [0,\infty)\to \R^+$ be non-decreasing functions such that
 \[\limsup_{x\to \infty} \frac{h(x)}{g(x)}<\infty.\]
Then $\sup_{n\geq 1}\E(g(|X_n|))<\infty$ implies that $\sup_{n\geq 1}\E(h(|X_n|))<\infty$.
\end{fact}

The concept of martingale and the following inequality are due to J.~L.~Doob, see e.g.\ \cite[Theorem~5.4.2]{D}.

\begin{theorem}[Doob's inequality]
Let $\{X_i: 1\leq i\leq n\}$ be a finite martingale difference sequence and let $p\geq 1$. Then for all $t>0$ we have 
\begin{equation*}
\P(M_n\geq t)\leq \frac{\E(|S_n|^p)}{t^p}.
\end{equation*}  
\end{theorem}

 \section{Technical lemmas} \label{s:tl} 
  
  \begin{lemma} \label{l:an} Let $\{a_n\}_{n\geq 1}$ be a positive, non-increasing sequence such that
  	\[\sum_{n=1}^{\infty} a_n<\infty.\]
  	Then there is a non-increasing sequence $\{b_n\}_{n\geq 1}$ such that
  	\begin{enumerate}[(1)]
  		\item \label{eq:l1} $b_n\geq a_n$ for all $n\geq 1$,
  		\item \label{eq:l2} $\lim_{n\to \infty} b_{n}/b_{n-1} =1$,
  		\item \label{eq:l3} $\sum_{n=1}^{\infty} b_n<\infty$.
  	\end{enumerate}
  \end{lemma}
  
  \begin{proof} Fix a positive sequence $\{c_k\}_{k\geq 1}$ such that $c_k \nearrow 1$. We can choose an increasing sequence of
  	positive integers $\{n_k\}_{k\geq 1}$
  	such that for all $k\in \N^+$ we have
  	\begin{equation} \label{ck2} \sum_{n=n_k}^{\infty} a_n<2^{-k}(1-c_{k+1})
  	\end{equation}
  	and
  	\begin{equation} \label{ck1} c_k^{n_{k+1}-n_{k}}\leq 2^{-k}(1-c_{k+1}).
  	\end{equation}
  	We will construct $b_n$ recursively. Let $b_n=a_n$ if $n\leq n_1$. For every $k\in \N^+$ and
  	$n_k< n\leq n_{k+1}$ define
  	\[b_{n}=\max\{a_{n}, c_k b_{n-1}\}.\]
  	Then clearly \eqref{eq:l1} holds and $b_{n}\leq b_{n-1}$ for all $n\leq n_1$.
  	Assume $n_k<n\leq n_{k+1}$ for some $k$. Then $a_n\leq a_{n-1}\leq b_{n-1}$ and our definition imply that
  	\[b_{n}=\max\{a_{n}, c_k b_{n-1}\}\leq \max\{b_{n-1}, c_k b_{n-1}\}=b_{n-1},\]
  	so $b_n$ is non-increasing.
  	
  	Now we show \eqref{eq:l2}. Assume that $n>n_k$. Let $n_m<n\leq n_{m+1}$ for some $m\geq k$. As the sequence $c_k$ is monotone increasing,
  	we have
  	\[c_k b_{n-1}\leq c_m b_{n-1} \leq b_n\leq b_{n-1},\]
  	so $c_k\leq b_n/b_{n-1}\leq 1$. Then $c_k \nearrow 1$ yields \eqref{eq:l2}.
  	
  	Finally, we prove \eqref{eq:l3}. For all $n\geq n_1$ define $d_{n,n}=a_n$ and for $i\geq n+1$ recursively define
  	\[d_{n,i}=c_k d_{n,i-1} \quad \textrm{if} \quad n_k< i\leq n_{k+1}.\]
  	Let $\ell \geq n_1$ be fixed. Let $n=n(\ell)$ be the largest integer such that $n_1\leq n \leq \ell$ and $b_n=a_n$. As $b_{n_1}=a_{n_1}$, we obtain that
  	$n$ exists. Then $b_{\ell}=d_{n,\ell}$ by our definitions. As the map $\ell \mapsto d_{n(\ell), \ell}$ is clearly one-to-one, we have
  	\begin{equation} \label{eq:dni} \sum_{n\geq n_1} b_n \leq \sum_{n\geq n_1} \sum_{i=n}^{\infty} d_{n,i}.
  	\end{equation}
  	Fix $k,n\in \N^+$ such that $n_k<n\leq n_{k+1}$. By definition
  	\begin{align} \label{eq:nk1}
  	\begin{split}
  	\sum_{i=n}^{n_{k+2}} d_{n,i}&=a_n\left(\sum_{i=0}^{n_{k+1}-n} c_k^i+c_{k}^{n_{k+1}-n} \sum_{i=1}^{n_{k+2}-n_{k+1}} c_{k+1}^i\right) \\
  	&\leq a_n\left(\frac{1}{1-c_{k}}+\frac{1}{1-c_{k+1}}\right)\leq \frac{2a_n}{1-c_{k+1}}.
  	\end{split}
  	\end{align}
  	For each $j\geq 2$ the definition of $d_{n,i}$ and \eqref{ck1} imply that
  	\begin{align} \label{eq:nk2}
  	\begin{split}
  	\sum_{i=n_{k+j}+1}^{n_{k+j+1}} d_{n,i}&=a_n c_{k}^{n_{k+1}-n} \left(\prod_{i=1}^{j-1} c_{k+i}^{n_{k+i+1}-n_{k+i}}\right)
  	\sum_{\ell=1}^{n_{k+j+1}-n_{k+j}} c_{k+j}^{\ell} \\
  	&\leq a_n c_{k+j-1}^{n_{k+j}-n_{k+j-1}}(1-c_{k+j})^{-1} \leq a_n 2^{-k-j+1}.
  	\end{split}
  	\end{align}
  	
  	By \eqref{eq:nk1} and \eqref{eq:nk2} for all $n\geq n_1$ we obtain
  	
  	\begin{equation} \label{eq:sumd} \sum_{i=n}^{\infty} d_{n,i}\leq \frac{2a_n}{1-c_{k+1}}+a_n \sum_{j=2}^{\infty} 2^{-k-j+1}\leq \frac{3a_n}{1-c_{k+1}}.
  	\end{equation}
  	Therefore \eqref{eq:dni}, \eqref{eq:sumd}, and \eqref{ck2} imply that
  	\[\sum_{n\geq n_1} b_n \leq \sum_{k=1}^{\infty} \sum_{n=n_k+1}^{n_{k+1}}  \sum_{i=n}^{\infty} d_{n,i}\leq \sum_{k=1}^{\infty} \sum_{n=n_k+1}^{n_{k+1}} \frac{3a_n}{1-c_{k+1}} \leq 3\sum_{k=1}^{\infty} 2^{-k}<\infty.\]
  	Thus \eqref{eq:l3} holds, and the proof is complete.
  \end{proof}
  
  \begin{corollary} \label{c:f} Let $g\colon [0,\infty)\to \R^+$ be a non-decreasing function such that
  	\[\sum_{n=1}^{\infty} \frac{1}{g(2^n)}<\infty.\]
  	Then there is a non-decreasing function $f\colon [0,\infty)\to \R^+$ such that
  	\begin{enumerate}[(i)]
  		\item \label{eq:c1} $\sum_{n=1}^{\infty} 1/f(2^n)<\infty$,
  		\item \label{eq:c2}  $\lim_{n\to \infty} f(2^{n+1})/f(2^n)=1$,  
  		\item \label{eq:c3} $\limsup_{x\to \infty} f(x)/g(x)\leq 1$.
  	\end{enumerate}
  \end{corollary}
  
  \begin{proof}  First we define $f(2^n)$ for all $n\in \N^+$. We apply Lemma~\ref{l:an} for the sequence
  	$a_n=1/g(2^n)$, let $\{b_n\}_{n\geq 1}$ be a sequence satisfying properties \eqref{eq:l1}--\eqref{eq:l3} in Lemma~\ref{l:an}.
  	Define $f(2^n)=1/b_n$ for all $n\in \N^+$. Then \eqref{eq:c1} holds by 
  	\eqref{eq:l3}, and by \eqref{eq:l2} for all $n$ we have
  	\begin{equation} \label{eq:clim} \lim_{n\to  \infty} \frac{f(2^{n+1})}{f(2^n)} =\lim_{n\to \infty} \frac{b_{n}}{b_{n+1}}=1,
  	\end{equation}
  	so \eqref{eq:c2} is satisfied. Since $b_n$ is non-increasing, the sequence $\{f(2^n)\}_{n\geq 1}$ is non-decreasing. 
  	Let $f\colon [0,\infty)\to \R^+$ be any non-decreasing function extending the sequence $\{f(2^n)\}_{n\geq 1}$. 
  	Clearly for all $n\in \N^+$ we have
  	\begin{equation} \label{eq:fbn} f(2^n)=\frac{1}{b_n}\leq \frac{1}{a_n}=g(2^n).
  	\end{equation}
  	Thus monotonicity, \eqref{eq:fbn}, and \eqref{eq:clim} imply that for all $n\in \N^+$ and $2^n\leq x\leq 2^{n+1}$ we have 	
  	\[\frac{f(x)}{g(x)}\leq \frac{f(2^{n+1})}{g(2^n)} \leq \frac{f(2^{n+1})}{f(2^n)} \to 1\]
  	as $n\to \infty$. The proof is complete. 
  \end{proof}

  \begin{lemma} \label{l:conc} Let $g\colon [0,\infty)\to \R^+$ be a non-decreasing function such that
  	\[\sum_{n=1}^{\infty} \frac{1}{g(2^n)}<\infty.\]
  	Then there is a non-decreasing function $f\colon [0,\infty)\to \R^+$ such that
  	\begin{enumerate}[(i)]
  		\item \label{eq:L1} $\sum_{n=1}^{\infty} 1/f(2^n)<\infty$,
  		\item \label{eq:L1.5} $\lim_{n\to \infty} f(2^{n+1})/f(2^n)=1$, 
  		\item \label{eq:L2} $\limsup_{x\to \infty} f(x)/g(x)\leq 1$,
  		\item \label{eq:L3} $f$ has continuous second derivative on $(0,\infty)$,
  		\item \label{eq:L4} $f'(x)/f(x)=o(1/x)$ as $x\to \infty$,
  		\item \label{eq:L5} $f''(x)/f(x)=o(1/x^2)$ as $x\to \infty$.
  	\end{enumerate}
  \end{lemma}
  
  \begin{proof} By Corollary~\ref{c:f} we may assume that 
  	\begin{equation} \label{eq:g} \lim_{n\to \infty} \frac{g(2^{n+1})}{g(2^n)}=1.
  	\end{equation} 
  	Let $f(2^n)=g(2^n)$ for all $n\in \N^+$, then clearly \eqref{eq:L1} and \eqref{eq:L1.5} hold. 
  	Each non-decreasing function $f\colon [0,\infty)\to \R^+$ extending the sequence $\{f(2^n)\}_{n\geq 1}$
  	satisfies \eqref{eq:L2}. Indeed, monotonicity, $f(2^{n+1})=g(2^{n+1})$, and \eqref{eq:g} imply that for every $n\in \N^+$ and 
  	$2^n\leq x\leq 2^{n+1}$ we have
  	\[\frac{f(x)}{g(x)}\leq \frac{f(2^{n+1})}{g(2^n)}=\frac{g(2^{n+1})}{g(2^n)} \to 1\]
  	as $n\to \infty$.
  	
  	Let $f(x)=f(2)$ for $0\leq x\leq 2$, and let $n\in \N^+$ be fixed.
  	For $0\leq x \leq 2^n$ define
  	\[r_n (x)=q_n (1-\cos(2^{1-n}\pi x)),\]
  	and let us define
  	\[f(2^n+x)=f(2^n)+\int_{0}^{x} r_n(t) \, \mathrm{d} t.\]
  	Then clearly
  	\[f(2^{n+1})-f(2^n)=\int_{0}^{2^{n}} r_n(t) \, \mathrm{d} t=2^n q_n ,\]
  	so
  	\begin{equation} \label{rn} q_n=2^{-n} (f(2^{n+1})-f(2^n)).
  	\end{equation}
  	Since $r_{n}(x)\geq 0$ for all $n\in \N^+$ and $0\leq x\leq 2^n$, the function $f$ is non-decreasing.
  	As $r_n(0)=r_n(2^n)=0$, we obtain that $f$ is continuously differentiable such that
  	$f'(x)=0$ if $0<x\leq 2$ and $f'(2^n+x)=r_n(x)$ for all $n\in \N^+$.
  	It is easy to see that $f''(2^n)=0$ for all $n\in \N^+$, so the formula of $f'(x)$ implies that $f'$ is continuously differentiable,
  	thus \eqref{eq:L3} holds. Let $n\in \N^+$ and $0\leq x \leq 2^n$. Then $r_n(x)\leq 2q_n$, \eqref{rn}, and \eqref{eq:L1.5} yield that
  	\begin{align*} \frac{f'(2^n+x)}{f(2^n+x)}&=\frac{r_{n}(x)}{f(2^n+x)} \leq 2^{1-n} \frac{f(2^{n+1})-f(2^n)}{f(2^n+x)} \\
  	&\leq \frac{4}{2^n+x} \frac{f(2^{n+1})-f(2^n)}{f(2^n)}=o\left(\frac{1}{2^n+x}\right)
  	\end{align*}
  	as $n\to \infty$, hence \eqref{eq:L4} is satisfied.
  	Let $n\in \N^+$ and $0\leq x \leq 2^n$. Clearly
  	\[|r_n'(x)|\leq 2^{1-2n} \pi (f(2^{n+1})-f(2^n)),\]
  	so \eqref{eq:L1.5} implies that
  	\begin{align*} \frac{|f''(2^n+x)|}{f(2^n+x)}&=\frac{|r'_{n}(x)|}{f(2^n+x)} \leq 2^{1-2n} \pi \frac{f(2^{n+1})-f(2^n)}{f(2^n+x)}
  	\\
  	&\leq \frac{8\pi}{(2^n+x)^2} \frac{f(2^{n+1})-f(2^n)}{f(2^n)}=o\left(\frac{1}{(2^n+x)^2}\right)
  	\end{align*}
  	as $n\to \infty$, so \eqref{eq:L5} holds. The proof is complete.
  \end{proof}
  
  \begin{corollary} \label{c:1}  Let $g\colon [0,\infty)\to \R^+$ be a non-decreasing function such that
  	\[\sum_{n=1}^{\infty} \frac{1}{g(2^n)}<\infty.\]
  	Then there is a non-decreasing function $f\colon [0,\infty)\to \R^+$ such that
  	\begin{enumerate}
  		\item \label{eq:h0} $\sum_{n=1}^{\infty} 1/f(2^n)<\infty$,
  		\item \label{eq:h0.5} $\lim_{n\to \infty} f(2^{n+1})/f(2^n)=1$, 
  		\item \label{eq:h1} $\limsup_{x\to \infty} f(x)/g(x)\leq 1$,
  		\item \label{eq:hc} for all $c>0$ there is an $R_c>0$ such that the function $h_c(x)=x^{-c} f(x)$ is decreasing for $x\geq R_c$,
  		\item \label{eq:gp} for all $0<p<1$ there is a 
  		concave increasing function $g_p\colon [0,\infty)\to \R^+$ and $N_p>0$ such that $g_p(x)=x^p f(x)$ for all $x\geq N_p$,
  		\item \label{eq:gq} for all $q>1$ there is a convex increasing function $g_q \colon [0,\infty)\to \R^+$ and $N_q>0$ such that $g_q$ is affine on $[0,N_q]$ and 
  		$g_q(x)=x^q f(x)$ for all $x\geq N_q$.
  	\end{enumerate}
  \end{corollary}
  
  \begin{proof}
  	Let us choose a non-decreasing function $f\colon [0,\infty) \to \R^+$ for which properties \eqref{eq:L1}--\eqref{eq:L5} of Lemma~\ref{l:conc} hold.
  	Then clearly $f$ satisfies \eqref{eq:h0}, \eqref{eq:h0.5}, and \eqref{eq:h1}. First we prove property \eqref{eq:hc}. By \eqref{eq:L4} of Lemma~\ref{l:conc} we have
  	\[h'_c(x)=-cx^{-c-1}f(x)+x^{-c} f'(x)=x^{-c-1}f(x)(-c+o(1))<0\]
  	if $x\geq R_c$ with some constant $R_c>0$, which proves 
  	\eqref{eq:hc}.
  	
  	Now we show \eqref{eq:gp}. Let $f_p(x)=x^p f(x)$, using \eqref{eq:L4} and \eqref{eq:L5} of Lemma~\ref{l:conc} we obtain that
  	\begin{align} \label{eq:Kp}
  	\begin{split}
  	f''_p(x)&=p(p-1)x^{p-2} f(x)+2px^{p-1} f'(x)+x^p f''(x) \\
  	&=x^{p-2} f(x)(p(p-1)+o(1))<0
  	\end{split}
  	\end{align}
  	if $x\geq K_p$ with some $K_p>0$.
  	By \eqref{eq:L4} of Lemma~\ref{l:conc} we obtain that
  	\begin{equation} \label{eq:Lp} \frac{f'_p(x)}{f_p(x)}=\frac{px^{p-1} f(x)+x^p f'(x)}{x^p f(x)}=\frac{x^{p-1}f(x)(p+o(1))}{x^p f(x)}=\frac{p+o(1)}{x}<\frac{1}{x}
  	\end{equation}
  	and $f'_p(x)>0$ if $x\geq L_p$ with some $L_p>0$. Let $N_p=\max\{K_p,L_p\}$.
  	Define $g_p(x)=f_p(x)$ if $x\geq N_p$ and let $g_p$ be affine on
  	$[0,N_p]$ with slope $f_p'(N_p)>0$. By \eqref{eq:Kp} we have $f_p''(x)<0$ for $x\geq N_p$, so
  	$g_p$ is increasing and concave. We only need to show that $g_p(0)>0$. Indeed, by \eqref{eq:Lp} we obtain that
  	\[g_p(0)=f_p(N_p)-f_p'(N_p)N_p>0.\]
  	Thus \eqref{eq:gp} holds.
  	
  	Finally, we prove \eqref{eq:gq}. Let $f_q(x)=x^q f(x)$, similarly to \eqref{eq:Kp} we obtain that
  	\begin{equation*}
  	f''_q(x)=x^{q-2} f(x)(q(q-1)+o(1))>0
  	\end{equation*}
  	and $f'_q(x)>0$ if $x\geq N_q$ with some $N_q>0$. 
  	Choose $0<\eps_q<f'_q(N_q)$ such that 
  	$f_q(N_q)-\eps_q N_q>0$. Let $g_q(x)=f_q(x)$ if $x\geq N_q$ and let $g_q$ be affine on $[0,N_q]$ with slope $\eps_q$. 
  	Clearly $g_q$ is increasing and convex, so we only need to show that $g_q(0)>0$. Indeed, we have 
  	\[g_q(0)=f_q(N_q)-\eps_q N_q>0.\]
  	Hence \eqref{eq:gq} holds, and the proof is complete.
  \end{proof}

  \begin{lemma} \label{l:c} Let $g\colon [0,\infty)\to \R^+$ be a non-decreasing function such that
  	\[\sum_{n=1}^{\infty} \frac{1}{g(2^n)}<\infty.\]
  	Then there is a non-decreasing function $f\colon [0,\infty)\to \R^+$ such that
  	\begin{enumerate}[(i)]
  		\item \label{eq:f1} $\sum_{n=1}^{\infty} 1/f(2^n)<\infty$,
  		\item \label{eq:f1.5} $\lim_{n\to \infty} f(2^{n+1})/f(2^n)=1$, 
  		\item \label{eq:f2} $\limsup_{x\to \infty} f(x)/g(x)<\infty$,
  		\item \label{eq:f3} $h(x)=xf(x)$ is piecewise linear, increasing, and convex on $[0,\infty)$.
  	\end{enumerate}
  \end{lemma}
  
  \begin{proof} By Corollary~\ref{c:f} we may assume that 
  	\begin{equation} \label{glim} \lim_{n\to \infty} \frac{g(2^{n+1})}{g(2^n)}=1.
  	\end{equation}  
  	Define $a_n=g(2^n)$ for all $n\in \N^+$. Define the sequence $\{b_n\}_{n\geq 1}$ for all $n\in \N^+$ such that
  	$b_1=a_1$, $b_2=a_2$, and for all $n\in \N^+$ we recursively define
  	\begin{equation} \label{bdef} b_{n+2}=\max\left \{a_{n+2}, (3/2)b_{n+1}-(1/2)b_n\right\}.
  	\end{equation} 
  	The definition clearly implies that for all $n\in \N^+$ we have
  	\begin{equation} \label{eq:2bn} 2b_{n+2}-3b_{n+1}+b_{n}\geq 0.
  	\end{equation}
  	First we show that $b_n$ is non-decreasing. Indeed, $b_2\geq b_1$ and assume by induction that
  	$b_{n+1}\geq b_{n}$ for some $n\geq 1$, then $b_{n+2}\geq (3/2)b_{n+1}-(1/2)b_n\geq b_{n+1}$.
  	Now we prove that for all $n\in \N^+$ we have
  	\begin{equation} \label{eq:ab} \frac{b_n}{a_n}\leq 2.
  	\end{equation}
  	Fix an arbitrary integer $m\geq 3$ and let $k$ be the largest integer such that
  	$2\leq k \leq m$ and $b_{k}=a_{k}$. As $b_2=a_2$, we obtain that $k$ exists.
  	Let us define $\{c_n\}_{n\geq 0}$ such that $c_0=b_{k-1}$, $c_1=b_{k}$, and for all $n\in \N$ let
  	\[c_{n+2}=(3/2)c_{n+1}-(1/2)c_{n}.\]
  	Then clearly $b_{m}=c_{m-k+1}$. Solving the linear recursion for $c_n$
  	and using that $0<c_0\leq c_1$ we obtain for all $n\in \N$ that
  	\[c_n=2c_1-c_0+\frac{c_0-c_1}{2^{n-1}}<2c_1.\]
  	Thus $b_{m}=c_{m-k+1}<2c_1=2b_{k}=2a_{k}$. Therefore the monotonicity of the sequence $\{a_n\}_{n\geq 1}$ implies that
  	\begin{equation} \label{eq:bn} \frac{b_{m}}{a_{m}}\leq 2\frac{a_{k}}{a_{m}}\leq 2,
  	\end{equation}
  	so \eqref{eq:ab} holds.
  	
  	Let us define $f\colon [0,\infty) \to \R^+$ as follows.
  	Let $f(2^n)=b_{n}$ for all $n\in \N^+$, and let $f(x)=f(2)$ if $0\leq x\leq 2$. If $n\in \N^+$ and
  	$0\leq x\leq 2^n$ then let
  	\[f(x+2^n)=b_n+2(b_{n+1}-b_n)\frac{x}{x+2^n}.\]
  	Since $\{b_n\}_{n\geq 1}$ is non-decreasing, it is easy to see that $f$ is non-decreasing and continuous.
  	Then $b_n\geq a_n$ and $\sum_{n=1}^{\infty} 1/a_n<\infty$ yield that \eqref{eq:f1} holds. 
  	
  	Let us define $e_n=b_{n+1}/b_{n}$ for all $n\in \N^+$ and let $E=\limsup_{n\to \infty} e_n$. 
  	Clearly $e_n\geq 1$ for all $n$, so it is enough to show for \eqref{eq:f1.5} that $E\leq 1$. 
  	By $b_{n+1}\geq a_{n+1}$ and \eqref{glim} we obtain that
  	\begin{equation} \label{en} \limsup_{n\to \infty} \frac{a_{n+2}}{b_{n+1}}\leq \limsup_{n\to \infty} \frac{a_{n+1}}{b_{n+1}}\lim_{n\to \infty} \frac{a_{n+2}}{a_{n+1}}\leq 1.
  	\end{equation}
  	For all $n\in \N^+$ we have 
  	\begin{equation} \label{recn} \frac{ (3/2)b_{n+1}-(1/2)b_n}{b_{n+1}}=\frac 32- \frac {1}{2e_n}\geq 1.
  	\end{equation} 
  	Then \eqref{bdef}, \eqref{en}, and \eqref{recn} yield that 
  	\begin{equation} \label{eq:E} E=\limsup_{n\to \infty} e_{n+1}\leq \limsup_{n\to \infty} \left(\frac 32-\frac{1}{2e_n}\right)=\frac 32- \frac{1}{2\limsup_{n} e_n}=\frac 32- \frac{1}{2E}.
  	\end{equation} 
  	Solving the above inequality implies that $E\leq 1$, so \eqref{eq:f1.5} is satisfied. 
  	
  	Monotonicity, \eqref{eq:bn}, and \eqref{glim} imply
  	that for all $n\in \N^+$ and $2^n\leq z \leq 2^{n+1}$ we have
  	\[\frac{f(z)}{g(z)}\leq \frac{f(2^{n+1})}{g(2^n)}=\frac{b_{n+1}}{a_n }\leq \frac{2a_{n+1}}{a_n} \to 2\]
  	as $n\to \infty$, so \eqref{eq:f2} holds.
  	
  	Finally, let us define $h\colon [0,\infty)\to [0,\infty)$ as $h(x)=xf(x)$. Then $h(x)=b_1 x$ if $0\leq x\leq 2$, and for all $n\in \N^+$ and $0\leq x\leq 2^n$ we have
  	\[h(x+2^n)=(x+2^n)f(x+2^n)=b_n 2^n+(2b_{n+1}-b_{n})x.\]
  	Clearly $h$ is continuous and increasing. We obtain that 
  	$h$ is affine on $[0,2]$ with slope $d_0:=b_1$, and for each $n\in \N^+$ it is also affine on $[2^n,2^{n+1}]$ with slope $d_{n}:=2b_{n+1}-b_{n}$, so $h$ is 
  	piecewise linear. In order to prove that $h$ is convex, we need to prove that the sequence $\{d_{n}\}_{n\geq 0}$ is non-decreasing. Clearly $d_1\geq d_0$ and by 
  	\eqref{eq:2bn} for all $n\in \N^+$ we have
  	\[d_{n+1}-d_{n}=2b_{n+2}-3b_{n+1}+b_{n}\geq 0.\]
  	Thus $\{d_{n}\}_{n\geq 0}$ is non-decreasing, so \eqref{eq:f3} holds. The proof is complete.
  \end{proof}
    
  \begin{lemma} \label{l:sqrt} Let $g\colon [0,\infty)\to \R^+$ be a non-decreasing function such that
  	\[\sum_{n=1}^{\infty} \frac{1}{g(2^n)}<\infty.\]
  	Let $q\geq 1$. Then there is a non-decreasing function $f\colon [0,\infty)\to \R^+$ such that
  	\begin{enumerate}[(i)]
  		\item \label{eq:sq1} $\sum_{n=1}^{\infty} 1/f(2^n)<\infty$,
  		\item \label{eq:sq2} $\lim_{n\to \infty} f(2^{n+1})/f(2^n)=1$, 
  		\item \label{eq:sq3} $\limsup_{x\to \infty} f(x)/g(x)<\infty$,
  		\item \label{eq:sq4} there is an increasing convex function $f_q\colon [0,\infty)\to \R^+$ and $N_q>0$ such that $f_q$ is affine on $[0,N_q]$ and 
  		$f_q(x)=x^q f(\sqrt{x})$ for $x\geq N_q$. 
  	\end{enumerate}
  \end{lemma}
  
  \begin{proof} Define $g^{*}\colon [0,\infty)\to \R^+$ as $g^{*}(x)=g(\sqrt{x})$. Clearly $g^*$ 
  	is a non-decreasing function such that $g^{*}(x)\leq g(x)$ for $x\geq 1$. Fact~\ref{f:1} yields that $\sum_{n=1}^{\infty} 1/g^{*}(2^n)<\infty$. 
  	Corollary~\ref{c:1} and Lemma~\ref{l:c} imply that there is a non-decreasing function $f^{*}\colon [0,\infty)\to \R^+$ such that 
  	\begin{enumerate}[(1)]
  		\item \label{eq:s1} $\sum_{n=1}^{\infty} 1/f^{*}(2^n)<\infty$, 
  		\item \label{eq:s2} $\lim_{n\to \infty} f^{*}(2^{n+1})/f^{*}(2^n)=1$,  	
  		\item \label{eq:s3} $\limsup_{x\to \infty} f^{*}(x)/g^{*}(x)<\infty$, 
  		\item \label{eq:s4} there is an increasing convex function $f^{*}_q \colon [0,\infty)\to \R^+$ and $N_q>0$ such that 
  		$f^*_{q}$ is affine on $[0,N_q]$ and $f^{*}_q(x)=x^q f^{*}(x)$ for all $x\geq N_q$.  
  	\end{enumerate}
  	Define $f\colon [0,\infty)\to \R^+$ as $f(x)=f^{*}(x^2)$. By \eqref{eq:s1} we have
  	\[\sum_{n=1}^{\infty} \frac{1}{f(2^n)}=\sum_{n=1}^{\infty} \frac{1}{f^*(4^n)}< \sum_{n=1}^{\infty} \frac{1}{f^*(2^n)}<\infty,\]
  	so \eqref{eq:sq1} holds. By \eqref{eq:s2} we have 
  	\[\lim_{n\to \infty}  \frac{f(2^{n+1})}{f(2^n)}=\lim_{n\to \infty} \frac{f^*(4^{n+1})}{f^{*}(4^{n})}=
  	\lim_{n\to \infty} \frac{f^*(2^{n+2})}{f^*(2^{n+1})} \cdot \lim_{n\to \infty} \frac{f^*(2^{n+1})}{f^*(2^n)}=1,\]
  	thus \eqref{eq:sq2} is satisfied. The definitions and \eqref{eq:s3} yield that 
  	\[\limsup_{x\to \infty} \frac{f(x)}{g(x)}=\limsup_{x\to \infty} \frac{f^*(x^2)}{g^*(x^2)}=\limsup_{x\to \infty} \frac{f^{*}(x)}{g^{*}(x)}<\infty,\]
  	so \eqref{eq:sq3} holds. Finally, define $f_q\colon [0,\infty)\to \R^+$ as $f_q(x)=f^{*}_q(x)$, then 
  	\eqref{eq:s4} yields that $f_q$ is an increasing convex function which is affine on $[0,N_q]$ and for all $x\geq N_q$ we have 
  	\[f_q(x)=f^{*}_q (x)=x^q f^*(x)=x^q f(\sqrt{x}),\] 
  	hence \eqref{eq:sq4} holds. The proof is complete.
  \end{proof}

\section{Sequences of arbitrary random variables} \label{s:arb}
The main goal of this section is to prove Theorems~\ref{t:arb1} and~\ref{t:arb2}. 
  
\begin{theorem} \label{t:arb1} Let $f\colon [0,\infty)\to \R^+$ be a non-decreasing function such that
	\[\sum_{n=1}^{\infty} \frac{1}{f(2^n)}<\infty.\]
	Let $0<r<1$ and let $r\leq p$, and define $q=q(r,p)=\max\{p, (p-r)/(1-r)\}$. 
	Assume that $\{X_n\}_{n\geq 1}$ is a sequence of random variables with $\sup_{n\geq 1} \E(|X_n|^q f(|X_n|))<\infty$. Then for all $\eps>0$ we have
	\[\sum_{n=1}^{\infty} n^{p/r-2} \P(M_n>\eps n^{1/r})<\infty.\]
\end{theorem}

\begin{proof} Let $\eps>0$ be arbitrarily fixed. We may assume that $X_i\geq 0$ for all $i$, otherwise we can replace $X_i$ by $|X_i|$. Thus $M_n=S_n$ for all $n\in \N^+$. First suppose that $p<1$, then $q=p$. 
We may assume by Corollary~\ref{c:1}~\eqref{eq:gp} and Fact~\ref{f:2} that there exists an $N_p>0$ and a concave increasing function $g_p\colon [0,\infty)\to \R^+$ such that $g_p(x)=x^p f(x)$ for all $x\geq N_p$. By Fact~\ref{f:2} we have
\begin{equation} \label{eq:pp}
\sup_{n\geq 1} \E g_p(|X_n|)=C<\infty.
\end{equation}
As $g_p$ is concave with $g_p(0)=0$, it is subadditive.
Fix an integer $n_0\geq (N_p/\eps)^r$. Markov's inequality, the fact that $g_p$ is increasing and subadditive, and \eqref{eq:pp} imply that for each
$\eps>0$ and $n\geq n_{0}$ we have
	\begin{align*} 
	\P(M_n>\eps n^{1/r})&=\P(|S_n|>\eps n^{1/r}) \\
	&= \P(g_p(|S_n|)>g_p(\eps n^{1/r})) \\
	&\leq \frac{\E g_p(|S_n|)}{g_p(\eps n^{1/r})} \\
	&= \frac{\E g_p(|S_n|)}{(\eps n^{1/r})^p f(\eps n^{1/r})} \\
	&\leq  \frac{\sum_{i=1}^{n} \E g_p(|X_i|)}{\eps^p n^{p/r} f(\eps n^{1/r})}  \\
	&\leq \frac{Cn}{\eps^p n^{p/r} f(\eps n^{1/r})} \\
	&\lesssim \frac{n^{1-p/r}}{f(\eps n^{1/r})}.
	\end{align*}
	Therefore by Fact~\ref{f:1} we have
	\[\sum_{n\geq n_{0}} n^{p/r-2} \P(M_n>\eps n^{1/r})\lesssim \sum_{n\geq n_{0}} \frac{1}{n f(\eps n^{1/r})}<\infty,\]
	which completes the proof for $p<1$. 
	
	Now assume that $p\geq 1$, then $q=(p-r)/(1-r)\geq 1$. We may assume by Corollary~\ref{c:1}~\eqref{eq:gq}, Lemma~\ref{l:c}, and Fact~\ref{f:2} that there is an
	$N_q>0$ and an increasing convex function $g_q\colon [0,\infty)\to \R^+$ such that $g_q(x)=x^q f(x)$ for all $x\geq N_q$. By Fact~\ref{f:2} we have
	\begin{equation} \label{eq:ssup} \sup_{n\geq 1} \E g_q(|X_n|)=C<\infty.
	\end{equation}
	Fix an integer $n_1\geq (N_q/\eps)^{r/(1-r)}$. Markov's inequality, the fact that $g_q$ is increasing and Jensen's inequality holds for the convex $g_q$, and \eqref{eq:ssup} imply that
	for each $\eps>0$ and $n\geq n_1 $ we have
	\begin{align*} \P(M_n>\eps n^{1/r})&=\P(|S_n|>\eps n^{1/r}) \\
	&\leq \P(g_q(|S_n|/n)>g_q(\eps n^{1/r-1})) \\
	&\leq \frac{\E g_q(|S_n|/n)}{g_q(\eps n^{1/r-1})} \\
	&=\frac{\E g_q(|S_n|/n)}{\eps^q n^{q(1/r-1)} f(\eps n^{1/r-1})} \\
	&\leq \frac{(1/n)\sum_{i=1}^{n} \E g_q(|X_i|)}{\eps^q n^{p/r-1} f(\eps n^{1/r-1})} \\
	&\leq \frac{C}{\eps^q n^{p/r-1} f(\eps n^{1/r-1})} \\
	&\lesssim \frac{n^{1-p/r}}{f(\eps n^{1/r-1})}.
	\end{align*} 
	Thus the above inequality and Fact~\ref{f:1} yields that
	\[\sum_{n\geq n_{1}} n^{p/r-2} \P(M_n>\eps n^{1/r})\lesssim \sum_{n\geq n_{1}} \frac{1}{n f(\eps n^{1/r-1})}<\infty.\]
	This concludes the proof.
\end{proof}

\begin{theorem} \label{t:arb2} Let $f\colon [0,\infty)\to \R^+$ be a non-decreasing function such that
\[\sum_{n=1}^{\infty} \frac{1}{f(2^n)}=\infty.\]
Let $0<r<1\leq p$ and let $q=q(r,p)=(p-r)/(1-r)$. Then there is a sequence of random variables $\{X_n\}_{n\geq 1}$ such that $\sup_{n\geq 1} \E(|X_n|^q f(|X_n|))<\infty$ and
\[\sum_{n=1}^{\infty} n^{p/r-2} \P(|S_n|>n^{1/r})=\infty.\]
\end{theorem}

\begin{proof} For all $k\in \N^+$ let
	\[p_k=\frac{4^{k(1-p/r)}}{f(4^{1+k(1/r-1)})}.\]
	Fix $k_0\in \N^{+}$ such that
	for all $k\geq k_0$ we have $p_k<1$.
	Let $X_n\equiv 0$ for all $n<4^{k_0}$. Let $k>k_0$ be fixed, then for all $m,n\in \{4^{k-1},\dots, 4^k-1\}$ let $X_n=X_m$ and let
	\[\P(X_n=4^{1+k(1/r-1)})=p_k \quad \textrm{and} \quad \P(X_n=0)=1-p_k.\]
	Then $\sup_{n\geq 1} \E(|X_n|^q f(|X_n|))=4^{q}<\infty$, and for all
	$2\cdot 4^{k-1}\leq n <4^k$ we have
	\begin{equation*} \label{eq:Sn} \P(|S_n|>n^{1/r})\geq \P(4^{k-1} X_n\geq 4^{k/r})=\P(X_n\geq 4^{1+k(1/r-1)})=p_k.
	\end{equation*}
	The above inequality and Fact~\ref{f:1} imply that
	\begin{align*} \sum_{n=1}^{\infty} n^{p/r-2} \P(|S_n|> n^{1/r}) &\geq \sum_{k>k_0} \sum_{2\cdot 4^{k-1}\leq n< 4^k} n^{p/r-2} \P(|S_n|>n^{1/r}) \\
	&\geq \sum_{k>k_0} (2\cdot 4^{k-1}) (4^{(k-1)(p/r)} 4^{-2k}) p_k \\
	&=2^{-2p/r-1} \sum_{k>k_0}  \frac{1}{f(4^{1+k(1/r-1)})}=\infty.
	\end{align*}
	The proof is complete.
\end{proof}

\section{Martingale difference sequences} \label{s:MDS}

\subsection{The cases $1<p<2$ and $r=p=1$}
 
The main goal of this subsection is to prove Theorem~\ref{t:MDS}.

 \begin{theorem} \label{t:MDS} Let $f\colon [0,\infty)\to \R^+$ be a non-decreasing function such that
 	\[\sum_{n=1}^{\infty} \frac{1}{f(2^n)}<\infty.\]
 	Let $1<p<2$ and $0<r\leq p$ or let $r=p=1$. Let $\{X_n\}_{n\geq 1}$ be a martingale difference sequence such that $\sup_{n\geq 1} \E(|X_n|^p f(|X_n|))<\infty$. Then for all $\eps>0$ we have \[\sum_{n=1}^{\infty} n^{p/r-2} \P(M_n>\eps n^{1/r})<\infty.\]
 \end{theorem}
 
 \begin{proof} Let $\eps>0$ be arbitrarily fixed and assume $\sup_{n\geq 1} \E(|X_n|^p f(|X_n|))=C<\infty$.
 	For all $k\in \N^+$ let 
 	\[F_k(t)=\P(|X_k|\leq t)\] 
 	be the cumulative distribution function of $|X_k|$.
 	Define $q=(2-p)/r>0$ and $c=2-p>0$. By Corollary~\ref{c:1}~\eqref{eq:hc} and Fact~\ref{f:2} we may assume that there exists an $R_c>0$ such that the function
 	$x\mapsto x^{-c} f(x)$ is decreasing for $x\geq R_c$. For all $n\in \N^+$ and for all $k\in \{1,\dots, n\}$ define
 	\begin{align*}
 	Y_{k,n}&=X_k I(|X_k|\leq n^{1/r})-\E(X_k I(|X_k|\leq n^{1/r}) \, |\, \iF_{k-1}), \\
 	Z_{k,n}&=X_k I(|X_k|> n^{1/r})-\E(X_k I(|X_k|> n^{1/r}) \, |\, \iF_{k-1}).
 	\end{align*}
 	For all $n\in \N^+$ define
 	\begin{align*} 
 	S_{k,n}^{*}&=\sum_{i=1}^{k} Y_{i,n} \quad \textrm{and} \quad
 	M_{n}^{*}=\max_{1\leq k\leq n} |S_{k,n}^{*}|, \\
 	S_{k,n}^{**}&=\sum_{i=1}^{k} Z_{i,n} \quad \textrm{and} \quad M_n^{**}=\max_{1\leq k\leq n} |S_{k,n}^{**}|.
 	\end{align*}  
 	Clearly for all $n\in \N^+$ and $1\leq k\leq n$ we have
 	\begin{equation} \label{e:YZ} 
 	\E(Y_{k,n} \, | \, \iF_{k-1})=\E(Z_{k,n}\, | \, \iF_{k-1})=0,
 	\end{equation}
 	so $\{S_{k,n}^{*}\}_{1\leq k\leq n}$ and $\{S_{k,n}^{**}\}_{1\leq k\leq n}$ are martingales. For all $n\in \N^+$ and $1\leq k\leq n$ we have 
 	\begin{equation*} 
 	S_k=S_{k,n}^{*}+S_{k,n}^{**}, 
 	\end{equation*}
 	so for all $n\in \N^+$ we have 
 	\begin{equation} \label{e:M} 
 	M_n\leq M_n^{*}+M_n^{**}.
 	\end{equation} 
 	By \eqref{e:YZ} we have $\E(Y_{k,n}Y_{\ell,n})=0$ for all $1\leq k<\ell\leq n$, so applying Doob's inequality for the martingale $\{S_{k,n}^{*}\}_{1\leq k\leq n}$ and the identity $\E(X-\E(X \, | \, \iF))^2=\E (X^2)-\E(\E(X \, | \, \iF))^2\leq \E (X^2)$ implies that
 	\begin{align*} \label{e:1}
 	\begin{split}
 	\P(M_n^{*}>\eps n^{1/r})&\leq \frac{1}{\eps^2} n^{-2/r} \E \left((S_{n,n}^{*})^2\right)  \\
 	&=\frac{1}{\eps^2} n^{-2/r} \sum_{k=1}^{n} \E (Y_{k,n}^2) \\
 	&\leq \frac{1}{\eps^2} n^{-2/r} \sum_{k=1}^{n} \E(X_k^2 I(|X_k|\leq n^{1/r})).
 	\end{split}
 	\end{align*}
 	Therefore
 	\begin{align*}
 	\sum_{n=1}^{\infty}   n^{p/r-2} \P(M_n^{*}>\eps n^{1/r}) &\lesssim
 	\sum_{n=1}^{\infty} n^{p/r-2/r-2}\sum_{k=1}^{n} \E(X_k^2 I(|X_k|\leq n^{1/r})) \\
 	&= \sum_{n=1}^{\infty} n^{-q-2} \sum_{k=1}^{n} \int_{0}^{\infty} t^2 I(t \leq n^{1/r})) \, \mathrm{d} F_k(t) \\
 	&=\sum_{k=1}^{\infty} \int_{0}^{\infty} t^2 \sum_{n\geq \max\{k,t^r\}} n^{-q-2}  \, \mathrm{d} F_k(t) \\
 	&\lesssim \sum_{k=1}^{\infty} \int_{0}^{\infty} t^2 (\max\{k,t^r\})^{-q-1}  \, \mathrm{d} F_k(t) \\
 	&=\sum_{k=1}^{\infty} (A_k+B_k+C_k),
 	\end{align*}
 	where
 	\[A_k=\int_{0}^{R_c} t^2 k^{-q-1} \, \mathrm{d} F_k(t),~
 	B_k=\int_{R_c}^{k^{1/r}} t^2 k^{-q-1} \, \mathrm{d} F_k(t),~ C_k=\int_{k^{1/r}}^{\infty} t^{p-r} \, \mathrm{d} F_k(t).\]
 	Clearly
 	\begin{equation} \label{e:A} \sum_{k=1}^{\infty} A_k\leq \sum_{k=1}^{\infty} R_c^2 k^{-q-1} <\infty.
 	\end{equation}
 	Let $k\geq R_c^r$. Using that  $x^{-c} f(x)$ is non-increasing if $x\geq R_c$, and $x^{r} f(x)$ is non-decreasing, we obtain that
 	\begin{align*} \label{kq1}
 	B_k+C_k&\leq k^{-q-1} \int_{R_c}^{k^{1/r}} t^2 \frac{t^{p-2} f(t)}{k^{(p-2)/r} f(k^{1/r})} \, \mathrm{d} F_k(t)+
 	\int_{k^{1/r}}^{\infty} t^{p-r} \frac{t^r f(t)}{k f(k^{1/r})} \, \mathrm{d} F_k(t) \notag \\
 	&\leq \frac{1}{k f(k^{1/r})} \int_{R_c}^{\infty} t^{p} f(t) \, \mathrm{d} F_k(t) \leq  \frac{\E(|X_k|^p f(|X_k|))}{k f(k^{1/r})}\leq \frac{C}{kf(k^{1/r})}.
 	\end{align*}
 	
 	By Fact~\ref{f:2} we have $A_k+B_k+C_k\leq \E(|X_k|^{p-r})<\infty$ for all $k$, so the above inequality with \eqref{e:A} and Fact~\ref{f:1} imply that
 	\begin{equation} \label{e:Sn1} \sum_{n=1}^{\infty}   n^{p/r-2} \P(M_n^{*}>\eps n^{1/r})\lesssim \sum_{k=1}^{\infty} (A_k+B_k+C_k)<\infty.
 	\end{equation}
 	Applying Doob's inequality for the martingale $\{S_{k,n}^{**}\}_{1\leq k\leq n}$, the triangle inequality, Jensen's inequality (for the conditional expectation as well), and the law of total expectation in this order implies that 
 	\begin{align*} 
 	\P(M_n^{**}>\eps n^{1/r})&\lesssim n^{-1/r} \E(|S_{n,n}^{**}|)  \\
 	&\leq n^{-1/r} \sum_{k=1}^{n} \E(|Z_{k,n}|) \\
 	&\leq 2n^{-1/r}\sum_{k=1}^{n} \E(|X_k| I(|X_k|> n^{1/r})) \\
 	&\leq 2n^{-1/r}\sum_{k=1}^{n} \frac{\E(|X_k| |X_k|^{p-1} f(|X_k|) I(|X_k|> n^{1/r}))}{n^{(p-1)/r}f(n^{1/r})} \\
 	&=2\frac{n^{-p/r}}{f(n^{1/r})} \sum_{k=1}^n 
 	\E(|X_k|^p f(|X_k|)) \\
 	&\leq 2C\frac{n^{1-p/r}} {f(n^{1/r})}.
 	\end{align*}
 	The above inequality and Fact~\ref{f:1} yield that 
 	\begin{equation} \label{e:Sn2} 
 	\sum_{n=1}^{\infty}   n^{p/r-2} \P(M_n^{**}>\eps n^{1/r})\lesssim \sum_{n=1}^{\infty} \frac{1}{nf(n^{1/r})}<\infty.
 	\end{equation}
 	
 	Finally, \eqref{e:M}, \eqref{e:Sn1}, and \eqref{e:Sn2} 
 	imply that
 	\begin{align*} \sum_{n=1}^{\infty}   n^{p/r-2} \P(M_n>\eps n^{1/r})&\leq \sum_{n=1}^{\infty}   n^{p/r-2} \P(M_n^{*}>(\eps/2) n^{1/r}) \\ &+\sum_{n=1}^{\infty}   n^{p/r-2} \P(M_n^{**}>(\eps/2)n^{1/r})<\infty.
 	\end{align*}
 	The proof is complete.
 \end{proof}
 
\subsection{The case $p\geq 2$}
 
The goal of the subsection is to prove Theorems~\ref{t:MDS2} and~\ref{t:cMDS}. 
 
 \begin{theorem} \label{t:MDS2} Let $f\colon [0,\infty)\to \R^+$ be a non-decreasing function such that
 	\[\sum_{n=1}^{\infty} \frac{1}{f(2^n)}<\infty.\]
 	Let $0<r<2\leq p$ and let $q=q(r,p)=2(p-r)/(2-r)$. Let $\{X_n\}_{n\geq 1}$ be a martingale difference sequence such that $\sup_{n\geq 1} \E(|X_n|^q f(|X_n|))<\infty$. Then for all $\eps>0$ we have \[\sum_{n=1}^{\infty} n^{p/r-2} \P(M_n>\eps n^{1/r})<\infty.\]
 \end{theorem}
 
Before proving the above theorem we need the next inequality due to Burkholder, Davis, and Gundy, see \cite[Theorem~1.1]{BDG} or \cite[Theorem~15.1]{B}.

\begin{theorem}[Burkholder--Davis--Gundy Inequality] \label{t:BDG} Let $g\colon [0,\infty)\to [0,\infty)$ be a convex function such that $g(0)=0$ and there is a constant $c\in \R^+$ such that
	\begin{equation} \label{eq:gc} g(2x)\leq cg(x) \quad \textrm{for all } x>0.
	\end{equation} 
	Then there exists a constant $C\in \R^+$ depending only on $c$ such that for every martingale difference sequence $\{X_i\}_{i\geq 1}$ for all $n\in \N^+$ we have 
	\[\E g(M_n)\leq C \E g\left(\sqrt{X_1^2+\dots+ X_n^2}\right).\]
\end{theorem}
 
 \begin{proof}[Proof of Theorem~\ref{t:MDS2}]
Let $\eps>0$ be arbitrarily fixed. 
As $q\geq 2$, by Lemma~\ref{l:sqrt} and Fact~\ref{f:2} we may assume that 
\begin{equation} \label{eq:f2n} \lim_{n\to \infty} \frac{f(2^{n+1})}{f(2^n)}=1
\end{equation} and there is an increasing convex function 
$f_{q/2}\colon [0,\infty)\to [0,\infty)$ and $N,a\in \R^+$ 
such that $f_{q/2}(0)=0$ and $f_{q/2}$ is linear on $[0,N]$, and for all $x\geq N$ we have 
\[f_{q/2}(x)=x^{q/2} f(\sqrt{x})-a.\] 
Define $g_q\colon [0,\infty)\to [0,\infty)$ as $g_q(x)=f_{q/2}(x^2)$. Since $g_q$ is a composition of convex increasing functions, it is increasing and convex. Clearly we have $g_q(0)=0$. Moreover, $g_q(x)=bx^2$ for $x\in [0,\sqrt{N}]$ with some constant $b>0$, 
and for $x\geq N$ we obtain 
\[g_q(x)=x^q f(x)-a.\] 
Fact~\ref{f:2} yields that 
\begin{equation} \label{eq:Cgq} 
\sup_{n\geq 1} \E g_q(|X_n|)=K<\infty.
\end{equation} 
Let $h_q\colon (0,\infty)\to \R^+$ be defined as 
\[h_q(x)=\frac{g_q(2x)}{g_q(x)}.\] 
As $g_q(x)>0$ for all $x>0$, the function $h_q$ is well defined, and the continuity of $g_q$ implies that $h_q$ is continuous, too. 
Since $g_q(x)=bx^2$ for $x\in [0,\sqrt{N}]$, we have $h_q(x)=4$ for $x\in (0,\sqrt{N}/2)$. By \eqref{eq:f2n} we obtain that $\limsup_{x\to \infty} h_q(x)<\infty$, 
so the continuity of $h_q$ implies that $h_q$ is bounded. Therefore $g_q$ satisfies the growth condition \eqref{eq:gc}. 

Let us choose $K\in \R^+$ such that $\eps^{q} f(\eps n^{1/r-1/2})>2a$ for each $n\geq K$. Define \[L=\max\left\{K,\left(\sqrt{N}/\eps\right)^{2r/(2-r)}\right\}.\] 
Applying that $g_q$ is increasing, Markov's inequality, Theorem~\ref{t:BDG} for $g_q$ and the finite martingale $\{S_i/\sqrt{n}\}_{1\leq i\leq n}$, 
Jensen's inequality for $f_{q/2}$, and \eqref{eq:Cgq} implies that for all $n\geq L$ we have   	
\begin{align*} \P(M_n>\eps n^{1/r})&=\P(g_q(M_n/\sqrt{n})> g_q(\eps n^{1/r-1/2})) \\
&\leq \frac{\E g_q(M_n/\sqrt{n})}{g_q(\eps n^{1/r-1/2})} \\
&\leq \frac{C\E\left(g_q\left(\sqrt{(1/n)\sum_{i=1}^{n} X_i^2}\right)\right)}{g_q(\eps n^{1/r-1/2})} \\
&=\frac{C\E\left(f_{q/2}\left((1/n)\sum_{i=1}^{n} X_i^2\right)\right)}{g_q(\eps n^{1/r-1/2})} \\
&\leq \frac{C(1/n)\sum_{i=1}^{n} \E f_{q/2}(X_i^2)}{g_q(\eps n^{1/r-1/2})} \\
&= \frac{C(1/n)\sum_{i=1}^{n} \E g_q(|X_i|)}{g_q(\eps n^{1/r-1/2})} \\
&\leq \frac{CK}{\eps^q n^{q(1/r-1/2)} f(\eps n^{1/r-1/2})-a} \\
&\leq \frac{2CK\eps^{-q} n^{1-p/r}}{f(\eps n^{1/r-1/2})} \\
&\lesssim  \frac{n^{1-p/r}}{f(\eps n^{1/r-1/2})}.
\end{align*}
Therefore the above inequality and Fact~\ref{f:1} imply that
\[\sum_{n\geq L} n^{p/r-2} \P(M_n>\eps n^{1/r})\lesssim \sum_{n\geq L}\frac{1}{n f(\eps n^{1/r-1/2})}<\infty.\]
The proof is complete. 
\end{proof}

\begin{theorem} \label{t:cMDS} Let $f\colon [0,\infty)\to \R^+$ be a non-decreasing function such that
	\[\sum_{n=1}^{\infty} \frac{1}{f(2^n)}=\infty.\]
Let $0<r<2\leq p$ and let $q=q(r,p)=2(p-r)/(2-r)$. Then there is a martingale difference sequence $\{X_n\}_{n\geq 1}$ such that 
$\sup_{n\geq 1} \E(|X_n|^q f(|X_n|))<\infty$ and 
\[\sum_{n=1}^{\infty} n^{p/r-2} \P(|S_n|>n^{1/r})=\infty.\]
\end{theorem}

\begin{proof} 
	Let $\{Y_n,Z_k\}_{n,k\geq 1}$ be independent random variables such that for all $n\in \N^+$ we have 
	\[\P(Y_n=1)=\P(Y_n=-1)=\frac 12.\] 
	For all $k\in \N^+$ let 
	\[p_k=\frac{4^{k(1-p/r)}}{f(4^{k(1/r-1/2)})}.\] 
	Fix $k_0\geq 2$ such that
	for all $k\geq k_0$ we have $p_k<1$. We define $Z_k\equiv 0$ if $k\leq k_0$ and for $k>k_0$ let 
	\[\P(Z_k=4^{k(1/r-1/2)})=p_k \quad \textrm{and} \quad \P(Z_k=0)=1-p_k.\]
	For all $k\in \N^+$ and $4^{k-1}\leq n<4^k$ let us define $X_n=Y_nZ_k$. Clearly we have 
	$\sup_{n\geq 1} \E(|X_n|^q f(|X_n|)=1$. Assume that $X_i\colon \Omega \to \R$ are random variables on 
	the probability space $(\Omega, \iF, \P)$. Let $\iF_0=\{\emptyset, \Omega\}$ and let $\iF_n=\sigma(X_1,\dots, X_n)$ for all $n\in \N^+$. 
	We show that $\{X_n\}_{n\geq 1}$ is a martingale difference sequence with respect to the natural filtration $\{\iF_n\}_{n\geq 1}$. 
	Fix $n,k\in \N^+$ with $4^{k-1}\leq n<4^k$. Indeed, as $Y_n$ is independent of $\{Z_1,\dots,Z_k,Y_1,\dots Y_{n-1}\}$, 
	it is independent of $\sigma(Z_k,\iF_{n-1})$, so a property of conditional expectation implies that for all $n\in \N^+$ we have
	\[\E(X_n \, | \, \iF_{n-1})=\E(Y_n Z_k \, | \, \iF_{n-1})=\E(Y_n) \E(Z_k \, |\, \iF_{n-1})=0,\]
so $\{X_n\}_{n\geq 1}$ is really a martingale difference sequence. By the central limit theorem there is an absolute constant $c>0$ such that for all $k>k_0$ and $2\cdot 4^{k-1} \leq n<4^{k}$ we have 
 \begin{equation} \label{eq:kisc} \P(Y_{4^{k-1}}+\dots+Y_{n}\geq 2^k)=\P(Y_{4^{k-1}}+\dots +Y_{n}\leq -2^k)\geq c.
 \end{equation} 
 We will prove that for all fixed $k>k_0$ and $2\cdot 4^{k-1} \leq n<4^{k}$ we have 
	\begin{equation} \label{eq:cpk} 
	\P(|S_n|>n^{1/r}) \geq c p_k.
	\end{equation} 
Let us use the notation $S=S_{4^{k-1}-1}$ and fix an arbitrary $x\in \R$ with $\P(S=x)>0$. By the law of total probability in order to prove \eqref{eq:cpk} 
it is enough to show that 
\begin{equation} \label{eq:cpk2} 
\P(|S_n|>n^{1/r} \, | \, S=x)\geq c p_k.
\end{equation}  
	As $4^{k/r}>n^{1/r}$, either $x+4^{k/r}>n^{1/r}$ or $x-4^{k/r}<-n^{1/r}$. We may assume by symmetry that $x+4^{k/r}>n^{1/r}$. It is clear from the definition that $S$ and $S_n-S$ are independent, and the independence of $Z_k$ and $\{Y_{4^{k-1}},\dots ,Y_n\}$, and \eqref{eq:kisc} yield that 
	\begin{align*} 	
	\P(|S_n|>n^{1/r} \, | \, S=x)&\geq \P(S_n-S\geq 4^{k/r} \, | \, S=x)\\
	&=\P(S_n-S\geq 4^{k/r}) \\
	&\geq \P(Y_{4^{k-1}}+\dots +Y_n\geq 2^k,~Z_k=4^{k(1/r-1/2)} )\\
	&=\P(Y_{4^{k-1}}+\dots +Y_{n}\geq 2^k) \P(Z_k=4^{k(1/r-1/2)})\\
	&\geq cp_k.
	\end{align*}
This implies \eqref{eq:cpk2}, so \eqref{eq:cpk} holds. Inequality \eqref{eq:cpk} and Fact~\ref{f:1} yield that
	\begin{align*} \sum_{n=1}^{\infty} n^{p/r-2} \P(|S_n|> n^{1/r}) &\geq \sum_{k>k_0} \sum_{2\cdot 4^{k-1}\leq n< 4^k} n^{p/r-2} \P(|S_n|>n^{1/r}) \\
	&\geq \sum_{k>k_0} (2\cdot 4^{k-1}) (4^{(k-1)(p/r)} 4^{-2k})c p_k \\
	&=c2^{-2p/r-1} \sum_{k>k_0}  \frac{1}{f(4^{k(1/r-1/2)})}=\infty.
	\end{align*}
	The proof is complete.
	\end{proof} 
	
 \section{Independent, negatively associated, and pairwise NQD random variables} \label{s:ind} 
 
 The main goal of this section is to prove 
 Theorems~\ref{t:indep} and~\ref{t:counter}.

 \begin{theorem} \label{t:indep} Let $f\colon [0,\infty)\to \R^+$ be a non-decreasing function such that 
 	\[\sum_{n=1}^{\infty} \frac{1}{f(2^n)}<\infty.\]
 	Let $0<r<2\leq p$ and let $\{X_n\}_{n\geq 1}$ be a sequence of negatively associated, centered random variables such that 
 	$\sup_{n\geq 1} \E(|X_n|^p f(|X_n|))<\infty$. For all $\eps>0$ we have \[\sum_{n=1}^{\infty} n^{p/r-2} \P(M_n>\eps n^{1/r})<\infty.\]
 \end{theorem} 
 
First we need the following inequality of Shao~\cite[Theorem~3]{Sh}.
 
 \begin{theorem}[Shao] \label{t:Shao}
 	Let $\{X_i: 1\leq i\leq n\}$ be a centered, negatively associated sequence of random variables with finite second moments. Let $M_n=\max_{1\leq k\leq n} |S_k|$ and $B_n=\sum_{i=1}^{n}\E(X_i^2)$. Then for all $x>0$, $a>0$, and $0<\alpha<1$ we have 
 \begin{align*} 
 \P(M_n\geq x)&\leq 2\P\left(\max_{1\leq i\leq n} |X_i|>a\right) \\
 &\quad+\frac{2}{1-\alpha} \exp\left(-\frac{x^2\alpha}{2(ax+B_n)} 
 \left(1+ \frac 23 \log \left(1+\frac{ax}{B_n}\right)\right)\right).
 \end{align*}
 \end{theorem}

 \begin{proof}[Proof of Theorem~\ref{t:indep}]
Fix $\eps>0$. By Fact~\ref{f:2} we have $\sup_{n\geq 1} \E(X_n^2)=C<\infty$. Thus $B_n=\sum_{i=1}^{n} \E(X_i^2)\leq Cn$ for all $n$. Let $N=8p/(2-r)$. 
Applying Theorem~\ref{t:Shao} for $n\in \N^+$, $x=\eps n^{1/r}$, $a=x/N$, and $\alpha=1/2$ we obtain that 
\begin{equation} \label{eq:AB} 
\P(M_n>\eps n^{1/r})\leq a_n+b_n,
\end{equation}
where 
\begin{align*} 
a_n&=2\P\left(\max_{1\leq i\leq n} |X_i|>\eps n^{1/r}/N\right), \\
b_n&=4 \exp \left(-\frac{\eps^2 n^{2/r}}{4(\eps^2 n^{2/r}/N+Cn)} \left(1+\frac{2}{3}
 \log \left( 1+\frac{\eps^2 n^{2/r}}{NCn}\right) \right) \right)
\end{align*}
Let $c=\eps/N$, by Markov's inequality we have 
\begin{align*} 
a_n&\leq 2\sum_{i=1}^{n} \P(|X_i|>cn^{1/r}) \\
&\leq  2\sum_{i=1}^{n}  \P(|X_i|^p f(|X_i|)\geq c^p n^{p/r}f(cn^{1/r})) \\
&\leq  2\sum_{i=1}^{n} \frac{\E(|X_i|^p f(|X_i|))}{c^p n^{p/r} f(cn^{1/r})} \\
&\lesssim \frac{n^{1-p/r}}{f(c n^{1/r})},
\end{align*}
thus Fact~\ref{f:1} implies that 
\begin{equation} \label{eq:A} 
\sum_{n=1}^{\infty} n^{p/r-2} a_n\lesssim 
\sum_{n=1}^{\infty} \frac{1}{n  f(c n^{1/r})}<\infty.
\end{equation}
Since $2/r>1$, easy calculation shows that 
\begin{equation*}
b_n=4\exp\left(\left(-\frac N6\left(\frac 2r -1\right)+o(1)\right)\log n \right) \quad \textrm{as } n\to \infty,
\end{equation*} 
so for all large enough $n$ we have 
\begin{equation*}
b_n\leq \exp(-(N/8)(2/r-1) \log n)=n^{-N(2-r)/(8r)}=n^{-p/r}.
\end{equation*}
Therefore 
\begin{equation} \label{eq:B} \sum_{n=1}^{\infty} n^{p/r-2} b_n<\infty.  
\end{equation} 
Clearly \eqref{eq:AB}, \eqref{eq:A}, and \eqref{eq:B} complete the proof. 
\end{proof} 
 
 \begin{corollary} \label{c:c} 
 	Let $f\colon [0,\infty)\to \R^+$ be a non-decreasing function such that 
 	\[\sum_{n=1}^{\infty} \frac{1}{f(2^n)}<\infty.\]
 	Let $0<r<1\leq p$, and let $\{X_n\}_{n\geq 1}$ be a sequence of 
 	\begin{enumerate}
 		\item pairwise NQD random variables if $1\leq p<2$,
 		\item negatively associated random variables if $p\geq 2$.
 	\end{enumerate} 
 	Assume that $\sup_{n\geq 1} \E(|X_n|^pf(|X_n|))<\infty$. Then for all $\eps>0$ we have \[\sum_{n=1}^{\infty} n^{p/r-2} \P(M_n>\eps n^{1/r})<\infty.\]
 \end{corollary}
 
 \begin{proof} Fix $\eps>0$ arbitrarily. We may assume that $X_n\geq 0$ almost surely for all $n$, otherwise we replace $X_n$ by $|X_n|$. Thus $M_n=S_n$. Fact~\ref{f:2} and $p\geq 1$ imply that $\sup_{n\geq 1} \E(X_n)=K<\infty$. For all $n\in \N^+$ define  
 	\[Y_n=X_n-\E(X_n).\]
 	Clearly if $\{X_n\}_{n\geq 1}$ is pairwise NQD/negatively associated then $\{Y_n\}_{n\geq 1}$ is also pairwise NQD/negatively associated. As $|Y_{n}|\leq \max\{K,X_n\}$, the monotonicity of the function $x\mapsto x^p f(x)$ implies that almost surely for all $n\in \N^+$ we have
 	\[|Y_n|^p f(|Y_n|)\leq K^p f(K)+X_n^pf(X_n),\]
 	so 
 	\[\sup_{n\geq 1} \E(|Y_n|^p f(|Y_n|))\leq 
 	K^p f(K)+\sup_{n\geq 1} \E(X_n^p f(X_n))<\infty.\]
 	
 	Define 
 	$T_n=\sum_{i=1}^{n} Y_i$. For all $n\geq (2K/\eps)^{r/(1-r)}$ we have 
 	\begin{align}
 	\begin{split} \label{s:st}  
 	\P(M_n>\eps n^{1/r})&=\P(S_n>\eps n^{1/r})\\ 
 	&\leq \P(T_n>\eps n^{1/r}-Kn) \\
 	&\leq \P(T_n>(\eps/2) n^{1/r}) \\
 	&\leq \P(|T_n|>(\eps/2) n^{1/r}).
 	\end{split}
 	\end{align}
 Applying Theorem~\ref{t:main}~\eqref{b1} if $p=1$, Theorem~\ref{t:pind} if $1< p<2$, and Theorem~\ref{t:indep} if $p\geq 2$ for $\{Y_n\}_{n\geq 1}$ yields that \[\sum_{n=1}^{\infty} n^{p/r-2} \P(|T_n|>(\eps/2) n^{1/r})<\infty,\] so \eqref{s:st} finishes the proof.  
 \end{proof}
 
 The following lemma is due to Nash~\cite{N}, 
 which gives a necessary and sufficient condition for $\P(\limsup_{n\to \infty} A_n)=1$ in terms of conditional probabilities.

 \begin{lemma}[Nash] \label{l:BC} Let $\{A_n\}_{n\geq 1}$ be events and define $\iH\subset \{0,1\}^{\N^+}$ such that 
 	\begin{align*} 
 	\iH=\{&(\alpha_1,\alpha_2,\dots ):\, \alpha_n=1 \textrm{ only for finitely many } n \\
 	&\textrm{and } \P(I(A_1)=\alpha_1,\dots, I(A_{n})=\alpha_{n})>0 \textrm{ for all } n\}.
 	\end{align*} 
 	Then $\P(\limsup_{n\to \infty} A_n)=1$ if and only if for all $(\alpha_1,\alpha_2,\dots)\in \iH$ we have 
 	\[\sum_{n=2}^{\infty} \P(A_n \, |\, I(A_1)=\alpha_1,\dots, I(A_{n-1})=\alpha_{n-1})=\infty.\]  
 \end{lemma}
 
The construction in the following theorem dates back to Chung~\cite[Theorem~2]{Ch}, but our proof is more involved. 

 \begin{theorem} \label{t:counter} Let $f\colon [0,\infty)\to \R^+$ be a non-decreasing function such that \[\sum_{n=1}^{\infty} \frac{1}{f(2^n)}=\infty.\]
 	Let $0<r<2$ and let $p\geq r$. Then there exists a sequence of independent, centered random variables $\{X_n\}_{n\geq 1}$
 	such that $\sup_{n\geq 1} \E(|X_n|^p f(|X_n|))<\infty$ and
 	\[\sum_{n=1}^{\infty} n^{p/r-2} \P(|S_n|> n^{1/r})=\infty.\]
 	Moreover, if $r=p$ then $\limsup_{n\to \infty} n^{-1/p}S_n\geq 1$ almost surely. 
 \end{theorem}

 \begin{proof} For all $k\in \N^+$ let
 	\[p_k=\frac{4^{-kp/r}}{f(4^{k/r})}.\]
 	Since $4^{-kp/r}\leq 4^{-k}$, for $c=\exp(-3/f(4^{1/r}))$ we can fix $k_0\in \N^{+}$ such that
 	for all $k\geq k_0$ we have $p_k<1/2$ and
 	\begin{equation} \label{eq:pk} (1-2p_k)^{4^k}\geq c.
 	\end{equation}
 	
 	We define a sequence of independent random variables $\{X_n\}_{n\geq 1}$
 	as follows. Let $X_n\equiv 0$ for all $n<4^{k_0}$. If $4^{k-1}\leq n<4^{k}$ for some integer $k>k_0$ then let
 	\[\P(X_n=4^{k/r})=\P(X_n=-4^{k/r})=p_k \quad \textrm{and} \quad \P(X_n=0)=1-2p_k.\]
 	Then $\{X_n\}_{n\geq 1}$ is a sequence of independent, centered random variables such that $\sup_{n\geq 1} \E(|X_n|^p f(|X_n|))=2$.
 	Fix $k>k_0$ and $2\cdot 4^{k-1} \leq n<4^{k}$. 
 	We will prove that 
 	\begin{equation} \label{eq:sn4} 
 	\P(|S_n|>n^{1/r}) \geq c4^{k-1} p_k.
 	\end{equation} 
 	Let us use the notation $S=S_{n-4^{k-1}}$ and fix an arbitrary $x\in \R$ with $\P(S=x)>0$. By the law of total probability in order to prove \eqref{eq:sn4} it is enough to show that 
 	\begin{equation} \label{eq:scond} 
 	\P(|S_n|>n^{1/r} \, | \, S=x)\geq c 4^{k-1} p_k.
 	\end{equation}  
 	As $4^{k/r}>n^{1/r}$, we have either $x+4^{k/r}>n^{1/r}$ or $x-4^{k/r}<-n^{1/r}$. By symmetry we may assume that $x+4^{k/r}>n^{1/r}$.  Thus the independence of $S$ and $S_n-S$, and \eqref{eq:pk} yield that 
 	\begin{align*} 	
 	\P(|S_n|>n^{1/r} \, | \, S=x)&\geq \P(S_n-S=4^{k/r} \, | \, S=x)=\P(S_n-S=4^{k/r}) \\
 	&\geq 4^{k-1} p_k (1-2p_k)^{4^{k-1}-1}\geq c4^{k-1}p_k,
 	\end{align*}
 	where only the sequences with $4^{k-1}-1$ zeros were taken into account.
 	This implies \eqref{eq:scond}, hence \eqref{eq:sn4} holds. 
 	Inequality~\eqref{eq:sn4} and Fact~\ref{f:1} yield that
 	\begin{align*} \sum_{n=1}^{\infty} n^{p/r-2} \P(|S_n|> n^{1/r}) &\geq \sum_{k>k_0} \sum_{2\cdot 4^{k-1}\leq n< 4^k} n^{p/r-2} \P(|S_n|>n^{1/r}) \\
 	&\geq \sum_{k>k_0} (2\cdot 4^{k-1}) (4^{(k-1)(p/r)} 4^{-2k})c4^{k-1} p_k \\
 	&=c2^{-2p/r-3} \sum_{k>k_0}  \frac{1}{f(4^{k/r})}=\infty.
 	\end{align*}
 	This proves the first claim. 
 	
 	Now assume that $p=r$. For all $k\in \N^+$ let us define the event 
 	\[A_k=\{S_{4^{k}-1}\geq 4^{k/p}\}.\]
 	Fix arbitrary $k>k_0$ and $(\alpha_1,\dots, \alpha_{k-1})\in \{0,1\}^{k-1}$ such that 
 	\[\P(I(A_1)=\alpha_1,\dots,I(A_{k-1})=\alpha_{k-1})>0.\]Repeating the argument of the proof of \eqref{eq:scond} for fixed values of $X_1,\dots, X_{4^{k-1}-1}$ and using the law of total probability for conditional probabilities we obtain that  
 	\begin{align}
 	\begin{split} \label{sp:I} 
 	\P(A_k \, | \, I(A_1)=\alpha_1,\dots, I(A_{k-1})=\alpha_{k-1})&\geq 3\cdot 4^{k-1}p_k (1-2p_k)^{3\cdot 4^{k-1}-1} \\
 	&\geq 3c4^{k-1} p_k\gtrsim \frac{1}{f(4^{k/p})}.
 	\end{split}
 	\end{align} 	
 	Fact~\ref{f:1} implies that $\sum_{k> k_0} 1/f(4^{k/p})=\infty$, so \eqref{sp:I} and Lemma~\ref{l:BC} yield that 
 	$\P(\limsup_{k\to \infty} A_k)=1$. Thus $\limsup_{n\to \infty} n^{-1/p} S_n\geq 1$ almost surely. The proof is complete.   
 \end{proof}

\end{document}